\documentclass[11pt,reqno]{amsart}

\usepackage[utf8]{inputenc}
\usepackage{setspace}
\usepackage{geometry}
\usepackage{enumerate}
\usepackage{enumitem, xcolor, amssymb,latexsym,amsmath,bbm}
\usepackage{mathtools}
\usepackage[mathscr]{euscript}
\usepackage{amsmath}
\usepackage{amssymb}
\usepackage{enumitem} 
\newcommand{\R}{\mathbb{R}}

\newcommand{\Ha}{\mathbb{H}}
\newcommand{\Z}{\mathbb{Z}}

\newcommand{\sys}{\mathrm{sys}}
\newcommand{\Sys}{\mathrm{Sys}}

\usepackage{tikz}
\usepackage{wrapfig}
\usepackage{enumerate}
\usepackage{graphicx}
\usepackage{subfigure}

\usetikzlibrary{arrows.meta}

\usetikzlibrary{decorations.markings}

\usepackage[colorlinks=true,citecolor=blue, linkcolor=blue,urlcolor=blue]{hyperref}

\calclayout
\usepackage{thmtools}

\theoremstyle{plain}
\newtheorem{theorem}{Theorem}[section]
\newtheorem{corollary}[theorem]{Corollary}
\newtheorem{lemma}[theorem]{Lemma}

\newtheorem{proposition}[theorem]{Proposition}
\theoremstyle{definition}
\newtheorem{definition}[theorem]{Definition}
\newtheorem{example}[theorem]{Example}
\theoremstyle{remark}

\usepackage[dvipsnames]{xcolor}

\numberwithin{equation}{section}
\numberwithin{figure}{section}
\allowdisplaybreaks

\title{Complexity in the Bolza Surface}
\author{}
\date{\today}
\author{Bhola Nath Saha}
\address{
Department of Mathematics and Statistics\\ 
Indian Institute of Technology  \\ 
Kanpur, Uttar Pradesh-208016\\
India}
\email{sahabholanath497@gmail.com}
\author{Bidyut Sanki}
\address{
Department of Mathematics and Statistics\\ 
Indian Institute of Technology  \\ 
Kanpur, Uttar Pradesh-208016\\
India}
\email{bidyut@iitk.ac.in}


\begin{document}

\begin{abstract}
     A surface in the Teichm\"uller space where the systole function attains its maximum, is called a maximal surface. For genus two there exists a unique maximal surface which is called the Bolza surface. In this article, we study the complexity of the set of systolic geodesics on the Bolza surface. We show that any non-systolic geodesic intersects the systolic geodesics in $2n$ points, where $n\geq 5$. For each non-negative integer $n$, we show the existence of curves on the Bolza surface which intersect the set of systolic geodesics at $(10+6n)$ and $(12+6n)$ points by construction. Furthermore, we show that there are exactly $12$ second systolic geodesics on the Bolza surface and they form a triangulation of the surface.
\end{abstract}
\maketitle
\tikzset{->-/.style={decoration={
  markings,
  mark=at position #1 with {\arrow{>}}},postaction={decorate}}}
  \tikzset{-<-/.style={decoration={
  markings,
  mark=at position #1 with {\arrow{<}}},postaction={decorate}}}
  \section{Introduction}
  A Riemann surface of constant curvature $-1$ is called a \emph{hyperbolic surface}. A compact surface with empty boundary is called a \emph{closed surface}. In this article, by a surface we always mean a closed orientable hyperbolic surface. For $g\geq 2$, let $\mathcal{T}_g$ denote the Teichm\"uller space of hyperbolic surfaces of genus $g$ (see Section 10.1 \cite{FarbMargalit}). A shortest essential closed geodesic on a surface $X\in \mathcal{T}_g$ is called a \emph{systolic geodesic} of $X$. The function $\sys:\mathcal{T}_g\to\R_+$, where $\sys(X)$ is the length of a systolic geodesic of $X$ (see Section 1~\cite{schaller1999systoles}), is called the systole function on $\mathcal{T}_g$ and $\sys(X)$ is called the systole of $X$.

 The function $\mathrm{sys}$ has been studied extensively by P. Schmutz~\cite{schaller1999systoles,Schmutz1993}, H. Akrout~\cite{Akrout}, and others. In Theorem A of \cite{schaller1999systoles}, Schmutz Schaller provided sufficient conditions for the systole function to be a topological Morse function and analyzed its critical points (also see \cite{Schmutz1993}). Subsequently, Akrout \cite{Akrout} established that $\sys$ is indeed a topological Morse function on Teichmüller space. It is an interesting and difficult problem to find the critical points of the systole function explicitly. These are the surfaces where $\sys$ attains a local maximum. In \cite{Jenni}, Jenni has proved that there is a unique surface in $\mathcal{T}_2$, where the systole function attains its global maximum. This surface is called Bolza surface and it can be obtained by identifying the opposite sides of a regular hyperbolic octagon with interior angle  $\frac{\pi}{4}$ (Section 5~\cite{Schmutz1993}). It has $12$ systolic geodesics and this is the genus two hyperbolic surface with the biggest automorphism group. In Theorem 3.1~\cite{MR3877282}, Sanki has given an explicit description of the systolic geodesics on the Bolza surface and shown that the set of all systolic geodesics forms a triangulation of the surface.
  
 The $k$-th complexity of a given system of closed geodesics on a surface is the number of simple closed geodesics that intersect the collection exactly at $k$ many points (see Definition \ref{def 2.1}). In Theorem 3.1 of \cite{MR3877282}, the author has shown that the $k$-th complexity of $\Sys (S)$ on the Bolza surface $S$ is zero, for $0\leq k\leq5$.  In this article, we further extend the study of the complexity of the systolic geodesics of the Bolza surface. We show that the $k$-th complexity of systolic geodesics is zero when either $k< 10$ or $k$ is an odd integer in the theorem below. 
 \begin{restatable}{theorem}{mainthmone}\label{thm:3.1}
      Let $\Sys(S)$ be the set of all systolic geodesics on the Bolza surface $S$. The $k$-th complexity $\mathcal{G}_k(\Sys(S))$ of $\Sys(S)$ is zero for all odd integers $k$ and even integers $k<10$.
 \end{restatable}

 The proof of Theorem \ref{thm:3.1} is based on a detailed analysis of the lifts of systolic geodesics and related curves to the hyperbolic upper half-plane, which serves as the universal cover of the Bolza surface.

 Theorem \ref{thm:3.1} naturally leads to the question of determining those values of $k$ for which the $k$-th complexity $\mathcal{G}_k(\Sys(S))$ is nonzero. We show that there exist infinitely many integers $k$ for which $\mathcal{G}_k(\Sys(S))\neq 0$. More precisely, we prove the following theorem.
 \begin{restatable}{theorem}{mainthmtwo}\label{thm:3.8}
    $\mathcal{G}_k(\Sys(S))\neq0$ for all $k$ of the form $(10+6n)$ and $(12+6n)$, where $n$ is a non-negative integer.
 \end{restatable}

The proof of Theorem \ref{thm:3.8} is based on an explicit construction of curves that intersect $\Sys(S)$ exactly $(10+6n)$ and $(12+6n)$ times, where $n$ is a non-negative integer. The key tool in these constructions is the use of appropriate Dehn twists.

 The set $\mathrm{Spec(X)}$ of lengths of all simple closed geodesics on a surface $X\in \mathcal{T}_g$ is called the \emph{length spectrum} of $X$. The length spectrum of a surface is a discrete, closed subset of $\R_+$ (see Lemma 12.4 \cite{FarbMargalit}) and hence its elements can be arranged in an ascending order
  $$\mathrm{Spec(X)}=\{l_i\mid l_{i+1}>l_i, \text{ for all } i\in \mathbf{N}\}.$$ 
 The $i$-th element $l_i$ of $\mathrm{Spec}(X)$ is called the $i$-th systole of $X$. A simple closed geodesic of length $l_i$ is called an \textit{$i$-th systolic geodesic}. Note that, $l_1=\sys(X)$. It is a very difficult problem to find the length spectrum of a given hyperbolic surface, in general. The length spectrum of the Bolza surface has been extensively studied by R. Aurich, F. Steiner and E. B. Bogomolny (see \cite{AurBogStein, AurStein}). In particular, the authors have evaluated the $i$-th systole on the Bolza surface, for certain values of $i$. In Section 2 of \cite{AurStein}, Aurich--Steiner have shown that the length of a second systolic geodesic in the Bolza surface is $2\cosh^{-1}{\left(3+2\sqrt{2}\right)}$. In this article, we completely determine all second systolic geodesics on the Bolza surface. More precisely, we prove the following theorem.
 
\begin{restatable}{theorem}{mainthmthree}\label{lem:4.1}
    There are exactly $12$ second systolic geodesics on the Bolza surface, and they form a triangulation with each triangle of type $(4,3,3)$.
\end{restatable}

\section{Preliminaries}
  In this section, we recall the notions of complexity of a system of curves on a surface, $(p,q,r)$-triangle and hyperelliptic surfaces. Also, we describe the isometries of the Bolza surface. These are essential for the subsequent sections.
  \begin{definition} \label{def 2.1}
  For a non-negative integer $k$, the \emph{$k$-th complexity} of a system of curves $\Omega$ on a surface $X$, denoted by $\mathcal{G}_k(\Omega)$, is the cardinality of the set below
  $$\left\{\gamma\in \mathcal{C}(X)\setminus \Omega \ \big|  \sum_{\alpha\in \Omega}i(\alpha,\gamma)= k\right\},$$
  where $\mathcal{C}(X)$ denotes the collection of all simple closed curves on $X$, pairwise distinct up to free homotopy.
  \end{definition}
  \begin{example}
    Let $\Sigma=\{\alpha,\gamma\}$, where $\alpha$ and $\gamma$ are as described in Figure \ref{a_sep_fill_on_S_2}. Then it is not difficult to find a simple closed curve $\delta$ that intersects $\gamma$ in $k$ points and is disjoint from $\alpha$. Consider the curves $T_\gamma^n(\delta)$, where $n\in \mathbb{Z}$ and $T_\gamma$ denotes the Dehn twist (see Chapter 3, \cite{FarbMargalit}) about $\gamma$. Then 
    $$i\bigl(T_\gamma^n(\delta), \Sigma\bigr)=k,$$
    for all $n\in\mathbb{Z}$ and hence $\mathcal{G}_k(\Sigma)=\infty,$ for all $k\in \mathbb{Z}_{\geq 0}$.
  
 Consider $\Omega=\{\alpha,\beta\}$, where $\alpha$ and $\beta$ are simple closed curves as shown in Figure~\ref{a_sep_fill_on_S_2}. We find the 0-th and the 1-st complexity of $\Omega$. Since these curves fill the surface, we have
$$\mathcal{G}_0(\Omega)=0.$$

To compute the first complexity of $\Omega$, consider the decomposition of the surface obtained by cutting along $\alpha\cup\beta$. This decomposition consists of a disjoint union of four polygons, as shown in Figure~\ref{the_disjoint_union_of_4_disks}. Consider the arc shown by a dashed line in the upper-left octagon, joining two points on the pair of edges labeled $y_2$. We choose the end points of the arc so that they are identified in the surface and therefore this arc projects to a simple closed curve on the surface that intersects $\Omega$ exactly once. Similarly, there is another such curve obtained by joining two points on the pair of edges labeled $y_5$ in the other octagon. These are the only simple closed curves that intersect $\Omega$ exactly once. Therefore,
$$\mathcal{G}_1(\Omega)=2.$$

  \begin{figure}[htbp]
    \begin{center}
    \begin{tikzpicture}[xscale=2,yscale=1.7]
        \draw  (0,0) ellipse (3cm and 1.2cm);
        
        \draw (-2,.05) to [bend right] (-1.2,.05);
        \draw (-2,.05) to [bend left] (-1.2,.05);
        
        \draw (1.2,.05) to [bend right] (2,.05);
        \draw (1.2,.05) to [bend left] (2,.05);
        
        \draw [red,-<-=.5] (0,1.2) arc
            [
                start angle=90,
                end angle=270,
                x radius=.3cm,
                y radius =1.2cm
            ] ;
        \draw [dashed, red] (0,-1.2) arc
            [
                start angle=270,
                end angle=450,
                x radius=.3cm,
                y radius =1.2cm
            ] ;
        \draw (-.1,.9) node {\tiny $A$};
        \draw (-.15,.48) node {\tiny $B$};
        \draw (-.16,-.25) node {\tiny $C$};
        \draw (-.15,-.75) node {\tiny $D$};
        \draw (.3,-1) node {\tiny $E$};
        \draw (.3,.95) node {\tiny $F$};
        
        \draw[red] (-.4,.3) node {$\alpha$};
        \draw[cyan,thick] (-1,.9) node {$\beta$};
        
        \draw [-<-=.35, cyan,thick] (2.3,0) arc
            [
                start angle=0,
                end angle=150,
                x radius=2.5cm,
                y radius =.8cm
            ] ;
        \draw[cyan,thick] (-2.36,.402).. controls (-3.2,0)and(-2.2,-.95)..(-1.2,-1.09);
        
        \draw [dashed, cyan,thick] (-1.2,-1.09) arc
            [
                start angle=265,
                end angle=360,
                x radius=2.2cm,
                y radius =1.12cm
            ] ;
        \draw [cyan,thick] (1.2,0) arc
            [
                start angle=0,
                end angle=260,
                x radius=1.7cm,
                y radius =.6cm
            ] ;
        \draw [cyan,thick] (-.8,-.59) arc
            [
                start angle=240,
                end angle=300,
                x radius=2.3cm,
                y radius =.5cm
            ] ;
        \draw[cyan,thick] (1.5,-.59)..controls (3,-.5) and (3,.5)..(1.45,1.045);
            ] ;
        \draw [dashed, cyan,thick] (1.45,1.045) arc
            [
                start angle=70,
                end angle=180,
                x radius=1.98cm,
                y radius =1.12cm
            ] ;
        \draw [cyan,thick] (-1.211,0) arc
            [
                start angle=180,
                end angle=360,
                x radius=1.752cm,
                y radius =.4cm
            ] ;

             \draw [red] (-1.8,-.05) arc
    [
    start angle=90,
    end angle=270,
    x radius=.12cm,
    y radius =.44cm
    ] ;
    \draw [red,dashed] (-1.8,-.93) arc
    [
    start angle=270,
    end angle=450,
    x radius=.12cm,
    y radius =.44cm
    ] ;
        
        \draw (1.1,.8) node {\footnotesize$y_1$};
        \draw (-1,-.32) node {\footnotesize$y_2$};
        \draw (2.73,0) node {\footnotesize$y_3$};
        \draw (-2.35,-.0) node {\footnotesize$y_4$};
        \draw (1.3,.2) node {\footnotesize$y_5$};
        \draw (-1.6,.78) node {\footnotesize$y_6$};
        
        \draw (-.3,1.1) node {\footnotesize$x_1$};
        \draw (.45,0) node {\footnotesize$x_2$};
        \draw (-.37,-.9) node {\footnotesize$x_3$};
        \draw (-.45,-.5) node {\footnotesize$x_4$};
        \draw (-.15,0) node {\footnotesize$x_5$};
        \draw (-.38,.67) node {\footnotesize$x_6$};

        \draw[red] (-2.05,-.5) node {\small$\gamma$};
    \end{tikzpicture}
    \end{center}
    \caption{Curves on $S_2$.} 
    \label{a_sep_fill_on_S_2}
    \end{figure}
     
    \begin{figure}[htbp]
    \begin{center}
    \begin{tikzpicture}[xscale=.75,yscale=.75]
        \draw[red,->-=.5] (2,0) .. controls (1.5705,0.656) .. (1.4142,1.4142);
        \draw[thick,cyan,->-=.5] (1.4142,1.4142).. controls (0.6778,1.5705) .. (0,2);
        \draw[red,->-=.5] (0,2).. controls (-0.6778,1.5705) .. (-1.4142,1.4142);
        \draw[thick,cyan,-<-=.5] (-1.4142,1.4142).. controls (-1.5705,0.656) .. (-2,0);
        \draw[red,->-=.5] (-2,0).. controls (-1.5705,-0.656) .. (-1.4142,-1.4142);
        \draw[thick,cyan,-<-=.5] (-1.4142,-1.4142).. controls (-0.6778,-1.5705) .. (0,-2);
        \draw[red,->-=.5] (0,-2).. controls (0.6778,-1.5705) .. (1.4142,-1.4142);
        \draw[thick,cyan,->-=.5] (1.4142,-1.4142).. controls (1.5705,-0.656) .. (2,0);
        
        \draw (2.05,0.656) node {\tiny $x_2$} (0.7,1.9) node {\tiny $y_6$} (-0.7,1.9) node {\tiny $x_1$} (-2,0.6) node {\tiny $y_2$} (-1.95,-0.75) node {\tiny $x_5$} (-0.9,-1.85) node {\tiny $y_4$} (0.9,-1.85) node {\tiny $x_4$} (1.95,-0.75) node {\tiny $y_2$};
        
        \draw[thick, cyan,-<-=.5] (9,0) .. controls (8.5705,0.656) .. (8.4142,1.4142);
        \draw[red,-<-=.5] ({7+1.4142},1.4142).. controls ({7+0.6778},1.5705) .. ({7+0},2);
        \draw[thick, cyan,-<-=.5] ({7+0},2).. controls ({7-0.6778},1.5705) .. ({7-1.4142},1.4142);
        \draw[red,-<-=.5] ({7-1.4142},1.4142).. controls ({7-1.5705},0.656) .. ({7-2},0);
        \draw[thick, cyan,->-=.5] ({7-2},0).. controls ({7-1.5705},-0.656) .. ({7-1.4142},-1.4142);
        \draw[red,-<-=.5] ({7-1.4142},-1.4142).. controls ({7-0.6778},-1.5705) .. ({7+0},-2);
        \draw[thick, cyan,->-=.5] ({7+0},-2).. controls ({7+0.6778},-1.5705) .. ({7+1.4142},-1.4142);
        \draw[red,-<-=.5] ({7+1.4142},-1.4142).. controls ({7+1.5705},-0.656) .. ({7+2},0);
        
        \draw ({7+2.05},0.656) node {\tiny $y_5$} ({7+0.7},1.9) node {\tiny $x_5$} ({7-0.7},1.95) node {\tiny $y_1$} ({7-2},0.6) node {\tiny $x_6$} ({7-1.95},-0.75) node {\tiny $y_5$} ({7-0.9},-1.85) node {\tiny $x_2$} ({7+0.9},-1.85) node {\tiny $y_3$} ({7+1.95},-0.75) node {\tiny $x_3$};
        
        \draw[-<-=.5,cyan,thick] (2,-5) .. controls (0.8,-4.2) .. (0,-3);
        \draw[->-=.5,red] (0,-3).. controls (-.8,-4.2) .. (-2,-5);
        \draw[->-=.5,cyan,thick](-2,-5).. controls (-.8,-5.8) .. (0,-7);
        \draw[->-=.5,red](0,-7) .. controls (0.8,-5.8) .. (2,-5);
        
        \draw (1.3,-4) node {\tiny $y_6$} (-1.2,-4) node {\tiny $x_3$} (-1,-6.2) node {\tiny $y_4$} (1.2,-6.2) node {\tiny $x_6$};
        
        \draw[-<-=.5,red] (9,-5) .. controls (7.8,-4.2) .. (7,-3);
        \draw[-<-=.5,cyan,thick] (7,-3).. controls (6.2,-4.2) .. (5,-5);
        \draw[-<-=.5,red](5,-5).. controls (6.2,-5.8) .. (7,-7);
        \draw[->-=.5,cyan,thick](7,-7) .. controls (7.8,-5.8) .. (9,-5);

        \draw[thick, dashed] (157.5:1.8) to (337.5:1.8);

        \draw (8.3,-4) node {\tiny $x_4$} (5.8,-4) node {\tiny $y_3$} (6,-6.2) node {\tiny $x_1$} (8.2,-6.2) node {\tiny $y_1$};
        
        \draw (0,2.2) node {\tiny A} (1.57,1.57) node{\tiny E} (2.2,0) node{\tiny F}(1.57,-1.57) node{\tiny C} (-.1,-2.2) node{\tiny B} (-1.57,-1.57) node{\tiny D} (-2.2,0) node{\tiny C} (-1.57,1.57) node{\tiny F};
        \draw (7,2.2) node {\tiny C} (8.57,1.57) node{\tiny B} (9.2,0) node{\tiny E}(8.57,-1.57) node{\tiny D} (6.9,-2.2) node{\tiny F} (5.43,-1.57) node{\tiny E} (4.8,0) node{\tiny B} (5.43,1.57) node{\tiny A};
        \draw (.1,-2.8) node {\tiny E} (2.2,-5) node{\tiny A} (0,-7.2) node{\tiny B} (-2.2,-5) node {\tiny D};
        \draw (7.1,-2.8) node {\tiny D} (9.2,-5) node{\tiny C} (7,-7.2) node{\tiny A} (4.8,-5) node {\tiny F};
    \end{tikzpicture}
    \end{center}
    \caption{Topological disks after cutting $S_2$ along $\alpha\cup\beta$.} 
    \label{the_disjoint_union_of_4_disks}
    \end{figure}

  \end{example}
  \begin{definition}
  Given positive integers $p,q,r$, satisfying $\frac{1}{p}+\frac{1}{q}+\frac{1}{r}<1$, a $(p,q,r)$-triangle is a hyperbolic triangle with angles $\frac{\pi}{p},\frac{\pi}{q},\frac{\pi}{r}$.
  \end{definition}
  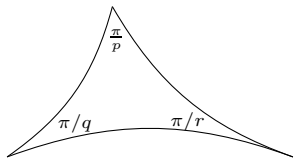
\begin{figure}[htbp]
      \centering
      \begin{tikzpicture}
          \draw (-.5,1) [bend left=20] to (-1.9,-1);
          \draw (-1.9,-1) [bend left=20] to (1.9,-1);
          \draw (1.9,-1) [bend left=20] to (-.5,1);
          \draw (-.45,.55) node {\tiny$\frac{\pi}{p}$};
          \draw (-1,-.55) node {\tiny$\pi/q$};
          \draw (.5,-.53) node {\tiny$\pi/r$};
      \end{tikzpicture}
      \caption{A hyperbolic triangle of type $(p,q,r)$.}
      \label{fig:enter-label}
  \end{figure}
  \subsection{Isometries of the Bolza surface}
Consider a regular hyperbolic octagon with each interior angle $\frac{\pi}{4}$ and equipped with a side paring as indicated in Figure \ref{fig:triang}. After identification of the sides in pairs, the resulting surface is the Bolza surface (see Section 5 \cite{Schmutz1993}). The rotation of the octagon about the center by an angle $\frac{\pi}{4}$ is side paring equivariant and hence induces an isometry of order eight, denoted by $R$.

\begin{figure}[htbp]
          \centering
          \begin{tikzpicture}
          \draw [bend left] ({2*cos(22.7)},{2*sin(22.7)}) to ({2*cos(67.5)},{2*sin(67.5)});
          \draw [bend left] ({2*cos(67.5)},{2*sin(67.5)}) to ({2*cos(112.5)},{2*sin(112.5)});
          \draw [bend left] ({2*cos(112.5)},{2*sin(112.5)}) to ({2*cos(157.5)},{2*sin(157.5)});
          \draw [bend left] ({2*cos(157.5)},{2*sin(157.5)}) to ({2*cos(202.5)},{2*sin(202.5)});
          \draw [bend left] ({2*cos(202.5)},{2*sin(202.5)}) to ({2*cos(247.5)}, {2*sin(247.5)});
          \draw [bend left] ({2*cos(247.5)}, {2*sin(247.5)}) to ({2*cos(292.5)},{2*sin(292.5)});
          \draw [bend left] ({2*cos(292.5)},{2*sin(292.5)}) to ({2*cos(337.5)},{2*sin(337.5)});
          \draw [bend left] ({2*cos(337.5)},{2*sin(337.5)}) to ({2*cos(22.7)},{2*sin(22.7)});

          \draw[bend left] ({1.62*cos(0)},{1.62*sin(0)}) to ({1.62*cos(45)},{1.62*sin(45)});
          \draw [bend left] ({1.62*cos(45)},{1.62*sin(45)}) to ({1.62*cos(90)},{1.62*sin(90)});
          \draw [bend left] ({1.62*cos(90)},{1.62*sin(90)}) to ({1.62*cos(135)},{1.62*sin(135)});
          \draw [bend left] ({1.62*cos(135)},{1.62*sin(135)}) to ({1.62*cos(180)},{1.62*sin(180)});
          \draw [bend left] ({1.62*cos(180)},{1.62*sin(180)}) to ({1.62*cos(225)},{1.62*sin(225)});
          \draw [bend left] ({1.62*cos(225)},{1.62*sin(225)}) to ({1.62*cos(270)},{1.62*sin(270)});
          \draw [bend left] ({1.62*cos(270)},{1.62*sin(270)}) to ({1.62*cos(315)},{1.62*sin(315)});
          \draw [bend left] ({1.62*cos(315)},{1.62*sin(315)}) to ({1.62*cos(0)},{1.62*sin(0)});

          \draw ({1.62*cos(0)},{1.62*sin(0)})--({1.62*cos(180)},{1.62*sin(180)});
          \draw ({1.62*cos(45)},{1.62*sin(45)})--({1.62*cos(225)},{1.62*sin(225)});
          \draw ({1.62*cos(90)},{1.62*sin(90)})--({1.62*cos(270)},{1.62*sin(270)});
          \draw ({1.62*cos(135)},{1.62*sin(135)})--({1.62*cos(315)},{1.62*sin(315)});

          \draw ({1.9*cos(-11.25)},{1.9*sin(-11.25)}) node {\tiny$c_1$} ({1.9*cos(11.25)},{1.9*sin(11.25)}) node {\tiny$c_2$};
          \draw ({1.9*cos(33.75)},{1.9*sin(33.75)}) node {\tiny$d_1$} ({1.9*cos(56.25)},{1.9*sin(56.25)}) node {\tiny$d_2$};
          \draw ({1.9*cos(78.75)},{1.9*sin(78.75)}) node {\tiny$a_2$} ({1.9*cos(101.25)},{1.9*sin(101.25)}) node {\tiny$a_1$};
          \draw ({1.9*cos(123.75)},{1.9*sin(123.75)}) node {\tiny$b_2$} ({1.9*cos(146.25)},{1.9*sin(146.25)}) node {\tiny$b_1$};
         \draw ({1.9*cos(168.75)},{1.9*sin(168.75)}) node {\tiny$c_2$} ({1.9*cos(191.25)},{1.9*sin(191.25)}) node {\tiny$c_1$};
         \draw ({1.9*cos(213.75)},{1.9*sin(213.75)}) node {\tiny$d_2$} ({1.9*cos(236.25)},{1.9*sin(236.25)}) node {\tiny$d_1$};
         \draw ({1.9*cos(258.75)},{1.9*sin(258.75)}) node {\tiny$a_1$} ({1.9*cos(281.25)},{1.9*sin(281.25)}) node {\tiny$a_2$};
         \draw ({1.9*cos(303.75)},{1.9*sin(303.75)}) node {\tiny$b_1$} ({1.9*cos(326.25)},{1.9*sin(326.25)}) node {\tiny$b_2$};

         \draw ({1.5*cos(22.5)},{1.5*sin(22.75}) node {\tiny$e_1$} ({1.5*cos(67.5)},{1.5*sin(67.5)}) node {\tiny$e_2$} ({1.5*cos(112.5)},{1.5*sin(112.5)}) node {\tiny$e_3$} ({1.5*cos(157.5)},{1.5*sin(157.5)}) node {\tiny$e_4$} ({1.5*cos(202.5)},{1.5*sin(202.5)}) node {\tiny$e_5$} ({1.5*cos(247.5)},{1.5*sin(247.5)}) node {\tiny$e_6$} ({1.5*cos(292.5)},{1.5*sin(292.5)}) node {\tiny$e_7$} ({1.5*cos(337.5)},{1.5*sin(337.5)}) node {\tiny$e_8$};

         \draw ({.9*cos(10)},{.9*sin(10)}) node {\tiny$f_1$} ({.9*cos(55)},{.9*sin(55)}) node {\tiny$f_2$} ({.9*cos(100)},{.9*sin(100)}) node {\tiny$f_3$} ({.9*cos(145)},{.9*sin(145)}) node {\tiny$f_4$} ({.9*cos(190)},{.9*sin(190)}) node {\tiny$f_5$} ({.9*cos(235)},{.9*sin(235)}) node {\tiny$f_6$} ({.9*cos(280)},{.9*sin(280)}) node {\tiny$f_7$} ({.9*cos(325)},{.9*sin(325)}) node {\tiny$f_8$};

         \draw[->-=1,thick] ({.4*cos(30)},{.4*sin(30)}) arc
    [
        start angle=30,
        end angle=270,
        x radius=.4cm,
        y radius =.4cm
    ] ;
          
          \end{tikzpicture}
          \caption{Triangulation of the octagon by $(4,4,4)$ triangles and isometry of order $8$.}
          \label{fig:triang}
\end{figure}
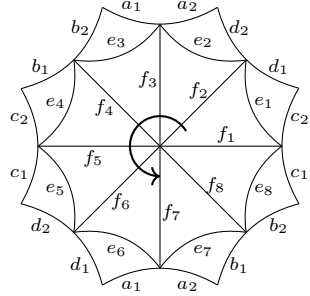

 Now, we describe an order $3$ isometry of the Bolza surface. For this, we consider the triangulated polygonal representation of the Bolza surface as given in Figure \ref{fig:order3}, where each triangle is of the form $(4,4,4)$. We note that this representation is obtained from the representation in Figure \ref{fig:triang} by cut and paste and therefore, they are equivalent. The rotations of the polygons about the centers by an angle $\frac{2\pi}{3}$ as indicated in Figure \ref{fig:order3} respect the side paring and hence induce an order $3$ isometry $L$ of the Bolza surface. It is straightforward to see that the isometries $R$ and $L$ generate the full orientation preserving isometry group of the Bolza surface. Moreover, this group is a von Dyck group (for details, see Chapter~5 of \cite{coxeter1980generators} and Section~1 of \cite{Stoytchev2020}).

\begin{figure}[htbp]
    \centering
    \begin{tikzpicture}[xscale=1,yscale=1]
        \draw (-1,0)--(-1.4,1.38)--(0,{sqrt(3)});
        \draw (1,0)--(1.4,1.38)--(0,{sqrt(3)});
        \draw[dashed] (-1,0)--(1,0)--(0,{sqrt(3)})--(-1,0);
        \draw (-1,0)--(0,-1.04)--(1,0);

        \draw (-1+5,0)--(-1.4+5,1.38)--(0+5,{sqrt(3)});
        \draw (1+5,0)--(1.4+5,1.38)--(0+5,{sqrt(3)});
        \draw[dashed] (-1+5,0)--(1+5,0)--(0+5,{sqrt(3)})--(-1+5,0);
        \draw (-1+5,0)--(0+5,-1.04)--(1+5,0);

        \draw (-1,0-4)--(-1.4,1.38-4)--(0,{sqrt(3)-4});
        \draw (1,0-4)--(1.4,1.38-4)--(0,{sqrt(3)-4});
        \draw[dashed] (-1,0-4)--(1,0-4)--(0,{sqrt(3)-4})--(-1,0-4);
        \draw (-1,0-4)--(0,-1.04-4)--(1,0-4);

        \draw (-1+5,0-4)--(-1.4+5,1.38-4)--(0+5,{sqrt(3)-4});
        \draw (1+5,0-4)--(1.4+5,1.38-4)--(0+5,{sqrt(3)-4});
        \draw[dashed] (-1+5,0-4)--(1+5,0-4)--(0+5,{sqrt(3)-4})--(-1+5,0-4);
        \draw (-1+5,0-4)--(0+5,-1.04-4)--(1+5,0-4);

        \draw (-.9,1.65) node {\tiny$e_2$} (.9,1.65) node {\tiny$d_1$} (1.4,.75) node{\tiny$c_2$} (.6,-.65) node{\tiny$e_8$} (-1.4,.75) node{\tiny$f_3$} (-.6,-.65) node{\tiny$f_8$} (.73,.86) node {\tiny$e_1$} (-.73,.86) node {\tiny$f_2$} (0,-.16) node {\tiny$f_1$};

        \draw (-.9+5,1.65) node {\tiny$c_1$} (.9+5,1.65) node {\tiny$d_1$} (1.4+5,.75) node{\tiny$e_6$} (.6+5,-.65) node{\tiny$f_3$} (-1.4+5,.75) node{\tiny$e_8$} (-.6+5,-.65) node{\tiny$f_4$} (.73+5,.86) node {\tiny$a_1$} (-.73+5,.86) node {\tiny$b_2$} (0+5,-.16) node {\tiny$e_3$};

        \draw (-.9,1.67-4) node {\tiny$f_4$} (.9,1.69-4) node {\tiny$f_7$} (1.4,.75-4) node{\tiny$e_6$} (.6,-.65-4) node{\tiny$d_2$} (-1.4,.75-4) node{\tiny$e_4$} (-.6,-.65-4) node{\tiny$c_1$} (.73,.86-4) node {\tiny$f_6$} (-.73,.86-4) node {\tiny$f_5$} (0,-.16-4) node {\tiny$e_5$};

        \draw (-.9+5,1.65-4) node {\tiny$f_8$} (.9+5,1.65-4) node {\tiny$e_4$} (1.4+5,.75-4) node{\tiny$c_2$} (.6+5,-.65-4) node{\tiny$d_2$} (-1.4+5,.75-4) node{\tiny$f_7$} (-.6+5,-.65-4) node{\tiny$e_2$} (.73+5,.86-4) node {\tiny$b_1$} (-.73+5,.86-4) node {\tiny$e_7$} (0+5,-.16-4) node {\tiny$a_2$};

        \draw (0,.57) node {\tiny$\bullet$} (0+5,.57) node {\tiny$\bullet$} (0,.57-4) node {\tiny$\bullet$} (0+5,.57-4) node {\tiny$\bullet$};

        \draw[->-=1,thick] ({0+.3*cos(30)},{.57+.3*sin(30)}) arc
    [
        start angle=30,
        end angle=270,
        x radius=.3cm,
        y radius =.3cm
    ] ;
    \draw[-<-=.1,thick] ({5+.3*cos(30)},{.57+.3*sin(30)}) arc
    [
        start angle=30,
        end angle=270,
        x radius=.3cm,
        y radius =.3cm
    ] ;
    \draw[->-=1,thick] ({0+.3*cos(30)},{.57+.3*sin(30)-4}) arc
    [
        start angle=30,
        end angle=270,
        x radius=.3cm,
        y radius =.3cm
    ] ;
    \draw[-<-=.1,thick] ({5+.3*cos(30)},{.57+.3*sin(30)-4}) arc
    [
        start angle=30,
        end angle=270,
        x radius=.3cm,
        y radius =.3cm
    ] ;
    \end{tikzpicture}
    \caption{Order $3$ isometry.}
    \label{fig:order3}
\end{figure}
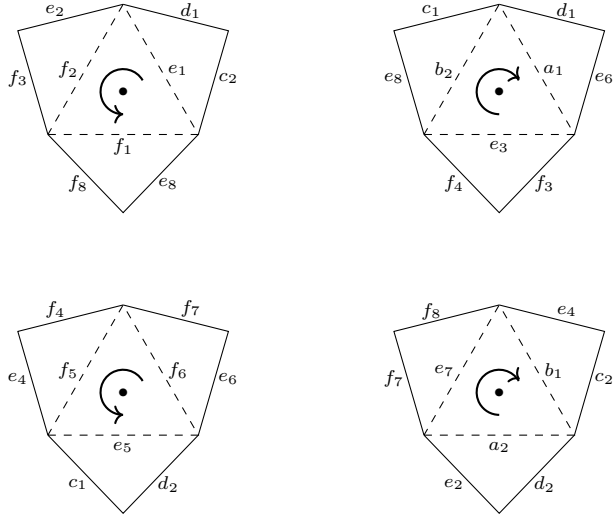

\subsection{Hyperelliptic involution in the Bolza surface} A \emph{hyperelliptic involution} of a Riemann surface $\Sigma$ of genus $g$ is a conformal automorphism $J:\Sigma\to \Sigma$ of order two with $2g+2$ fixed points. The fixed points are the \emph{Weierstrass points}. A surface $\Sigma$ admitting a hyperelliptic involution is called a \emph{hyperelliptic surface}. Note that, every hyperbolic surface of genus two is hyperelliptic. For more details, we refer the reader to Section 3.7 of \cite{FarkasKra}.

The rotation of the octagon (see Figure \ref{fig:involn}) about the center by an angle $\pi$ induces the hyperelliptic involution of the Bolza surface. We recall the result below due to Haas-Susskind \cite{Haas} which is essential for the subsequent sections.

\begin{figure}[htbp]
    \centering
    \begin{tikzpicture}
           \draw [bend left] ({2*cos(22.7)},{2*sin(22.7)}) to ({2*cos(67.5)},{2*sin(67.5)});
          \draw [bend left] ({2*cos(67.5)},{2*sin(67.5)}) to ({2*cos(112.5)},{2*sin(112.5)});
          \draw [bend left] ({2*cos(112.5)},{2*sin(112.5)}) to ({2*cos(157.5)},{2*sin(157.5)});
          \draw [bend left] ({2*cos(157.5)},{2*sin(157.5)}) to ({2*cos(202.5)},{2*sin(202.5)});
          \draw [bend left] ({2*cos(202.5)},{2*sin(202.5)}) to ({2*cos(247.5)}, {2*sin(247.5)});
          \draw [bend left] ({2*cos(247.5)}, {2*sin(247.5)}) to ({2*cos(292.5)},{2*sin(292.5)});
          \draw [bend left] ({2*cos(292.5)},{2*sin(292.5)}) to ({2*cos(337.5)},{2*sin(337.5)});
          \draw [bend left] ({2*cos(337.5)},{2*sin(337.5)}) to ({2*cos(22.7)},{2*sin(22.7)});

          \draw[green] ({2*cos(22.7)},{2*sin(22.7)}) node {\tiny$\bullet$} ({2*cos(67.5)},{2*sin(67.5)}) node {\tiny$\bullet$} ({2*cos(112.5)},{2*sin(112.5)}) node {\tiny$\bullet$} ({2*cos(157.5)},{2*sin(157.5)}) node {\tiny$\bullet$} ({2*cos(202.5)},{2*sin(202.5)}) node {\tiny$\bullet$} ({2*cos(247.5)}, {2*sin(247.5)}) node {\tiny$\bullet$} ({2*cos(292.5)},{2*sin(292.5)}) node {\tiny$\bullet$} ({2*cos(337.5)},{2*sin(337.5)}) node {\tiny$\bullet$};
          
          \draw ({1.93*cos(0)},{1.93*sin(0)}) node {$c$} ({1.9*cos(45)},{1.9*sin(45)}) node {$d$} ({1.9*cos(90)},{1.9*sin(90)}) node {$a$} ({1.9*cos(135)},{1.9*sin(135)}) node {$b$} ({1.9*cos(180)},{1.9*sin(180)}) node {$c$} ({1.9*cos(225)},{1.9*sin(225)}) node {$d$} ({1.9*cos(270)},{1.9*sin(270)}) node {$a$} ({1.93*cos(315)},{1.93*sin(315)}) node {$b$};

          \draw[pink] ({1.62*cos(0)},{1.62*sin(0)}) node {\tiny$\bullet$} ({1.63*cos(180)},{1.63*sin(180)}) node {\tiny$\bullet$};
          \draw[blue] ({1.62*cos(45)},{1.62*sin(45)}) node {\tiny$\bullet$} ({1.64*cos(225)},{1.64*sin(225)}) node {\tiny$\bullet$};
          \draw[purple,thick] ({1.63*cos(90)},{1.63*sin(90)}) node {\tiny$\bullet$} ({1.63*cos(270)},{1.63*sin(270)}) node {\tiny$\bullet$};
          \draw[cyan] ({1.63*cos(135)},{1.63*sin(135)}) node {\tiny$\bullet$} ({1.63*cos(315)},{1.63*sin(315)}) node {\tiny$\bullet$}; 

           \draw[->-=1,thick] ({.7*cos(30)},{.7*sin(30)}) arc
    [
        start angle=30,
        end angle=270,
        x radius=.7cm,
        y radius =.7cm
    ] ;
    \draw (0,0) node {\tiny$\bullet$};
    \draw (-.8,.5) node {$\pi$};

    \end{tikzpicture}
    \caption{Hyperelliptic involution of the Bolza Surface. $\bullet$ represents the fixed points of the involution.}
    \label{fig:involn}
\end{figure}
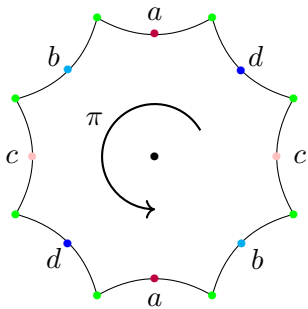

\begin{theorem}[Theorem 1~\cite{Haas}]\label{thm2.4}
     Let $J$ be the hyperelliptic involution of a genus two Riemann surface $M$. Then every simple closed geodesic on $M$ is mapped onto itself by $J$. Furthermore, suppose $\alpha$ is a simple closed geodesic on $M$. If $\alpha$ is a separating curve then $J$ preserves the orientation of $\alpha$. Otherwise, $J$ reverses the orientation of $\alpha$.
\end{theorem}

  \section{Complexity}
  
  In this section, we show that the odd complexity of the set of systolic geodesics $\Sys(S)$ is zero. Also, we show that the even complexity is finite and it is non-zero for all integers of the form $(10+6n)$ and $(12+6n)$, $n\in\Z_{\geq0}$. We begin by proving Theorem \ref{thm:3.1}.
  
\mainthmone*

  
  In Section 2.1, we have seen that the Bolza surface is obtained by identifying the opposite sides of a regular hyperbolic octagon with each interior angle $\frac{\pi}{4}$. There is an alternative construction of the Bolza surface (Section 3 of ~\cite{MR3877282}), by taking two identical copies of the right angled regular hyperbolic octagon and identifying the sides in pairs as shown in Figure \ref{fig1.1}. The sides $\{a_1,a_2\},\{b_1,b_2\}$, $\{ c_1 ,c_2 \}$ and $\{d_1,d_2\}$ project onto simple closed geodesics on $S$. We denote these geodesics by $a, b,c,d$ respectively and define $\Omega_1=\{a, b,c,d\}$. Also, the diagonals joining opposite vertices of each octagon project onto simple closed geodesics. We define $\Omega_2$ as the collection of all such geodesics. In~\cite{MR3877282}, Sanki has proved that $\Sys(S)=\Omega_1\cup \Omega_2$ is the set of all systolic geodesics in $S$. 

  Let $p:\mathbb{D} \to S$ be the universal covering map, where $\mathbb{D}$ is the Poincar\'e disc model of the hyperbolic plane. Then $p^{-1}(\Omega_1)$ is a tessellation $\mathbb{T}$ of $\mathbb{D}$ by regular right angled octagons. Each tile of $\mathbb{T}$ maps under $p$ to one of the octagons $P_1$ or $P_2$, as in Figure~\ref{fig1.1}. A tile of $\mathbb{T}$ is called of Type I (or Type II) if it maps to $P_1$ (or $P_2$ respectively).

  We regard $\Omega_1$ as a graph, where the vertices are the intersection points of the systolic geodesics in $\Omega_1$ and the geodesic segments between the vertices are the edges.
  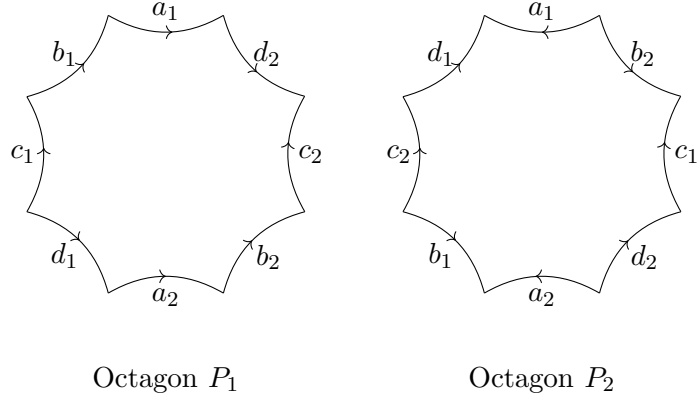
\begin{figure}[htbp]
    \centering
    \begin{tikzpicture}
         \draw [bend left,-<-=.5] ({2*cos(22.7)},{2*sin(22.7)}) to ({2*cos(67.5)},{2*sin(67.5)});
         \draw [bend left,-<-=.5] ({2*cos(67.5)},{2*sin(67.5)}) to ({2*cos(112.5)},{2*sin(112.5)});
          \draw [bend left,-<-=.5] ({2*cos(112.5)},{2*sin(112.5)}) to ({2*cos(157.5)},{2*sin(157.5)});
          \draw [bend left,-<-=.5] ({2*cos(157.5)},{2*sin(157.5)}) to ({2*cos(202.5)},{2*sin(202.5)});
          \draw [bend left,->-=.5] ({2*cos(202.5)},{2*sin(202.5)}) to ({2*cos(247.5)}, {2*sin(247.5)});
          \draw [bend left,->-=.5] ({2*cos(247.5)}, {2*sin(247.5)}) to ({2*cos(292.5)},{2*sin(292.5)});
          \draw [bend left,->-=.5] ({2*cos(292.5)},{2*sin(292.5)}) to ({2*cos(337.5)},{2*sin(337.5)});
          \draw [bend left,->-=.6] ({2*cos(337.5)},{2*sin(337.5)}) to ({2*cos(22.7)},{2*sin(22.7)});

         \draw [bend left,-<-=.5] ({5+2*cos(22.7)},{2*sin(22.7)}) to ({5+2*cos(67.5)},{2*sin(67.5)});
         \draw [bend left,->-=.5] ({5+2*cos(67.5)},{2*sin(67.5)}) to ({5+2*cos(112.5)},{2*sin(112.5)});
         \draw [bend left,-<-=.5] ({5+2*cos(112.5)},{2*sin(112.5)}) to ({5+2*cos(157.5)},{2*sin(157.5)});
          \draw [bend left,-<-=.5] ({5+2*cos(157.5)},{2*sin(157.5)}) to ({5+2*cos(202.5)},{2*sin(202.5)});
          \draw [bend left,->-=.5] ({5+2*cos(202.5)},{2*sin(202.5)}) to ({5+2*cos(247.5)}, {2*sin(247.5)});
          \draw [bend left,-<-=.5] ({5+2*cos(247.5)}, {2*sin(247.5)}) to ({5+2*cos(292.5)},{2*sin(292.5)});
          \draw [bend left,->-=.5] ({5+2*cos(292.5)},{2*sin(292.5)}) to ({5+2*cos(337.5)},{2*sin(337.5)});
         \draw [bend left,->-=.6] ({5+2*cos(337.5)},{2*sin(337.5)}) to ({5+2*cos(22.7)},{2*sin(22.7)});

         \draw ({1.93*cos(0)},{1.93*sin(0)}) node {$c_2$} ({1.9*cos(45)},{1.9*sin(45)}) node {$d_2$} ({1.9*cos(90)},{1.9*sin(90)}) node {$a_1$} ({1.9*cos(135)},{1.9*sin(135)}) node {$b_1$} ({1.9*cos(180)},{1.9*sin(180)}) node {$c_1$} ({1.9*cos(225)},{1.9*sin(225)}) node {$d_1$} ({1.9*cos(270)},{1.9*sin(270)}) node {$a_2$} ({1.93*cos(315)},{1.93*sin(315)}) node {$b_2$};

         \draw ({5+1.93*cos(0)},{1.93*sin(0)}) node {$c_1$} ({5+1.9*cos(45)},{1.9*sin(45)}) node {$b_2$} ({5+1.9*cos(90)},{1.9*sin(90)}) node {$a_1$} ({5+1.9*cos(135)},{1.9*sin(135)}) node {$d_1$} ({5+1.9*cos(180)},{1.9*sin(180)}) node {$c_2$} ({5+1.9*cos(225)},{1.9*sin(225)}) node {$b_1$} ({5+1.9*cos(270)},{1.9*sin(270)}) node {$a_2$} ({5+1.93*cos(315)},{1.93*sin(315)}) node {$d_2$};
         \draw (0,-3) node {Octagon $P_1$} (5,-3) node{Octagon $P_2$};

    \end{tikzpicture}
    \caption{Alternative description of the Bolza Surface.}
    \label{fig1.1}
    \end{figure}

To simplify the proof of Theorem \ref{thm:3.1}, we first prove the following four lemmas (Lemmas \ref{lemma:16geodesics}--\ref{lemma 4.4}). Their combination yields the desired result.

    \begin{lemma}\label{lemma:16geodesics}
        The set $\Gamma$ of all non-systolic geodesics in $S$, intersecting the graph $\Omega_1$ only at the vertices, has $16$ elements.
    \end{lemma}
    \begin{proof}
        Suppose $\delta\in \Gamma$ and  $\Tilde{\delta}$ denotes its path lift in $\mathbb D$ with the initial point at some vertex. Then $\Tilde{\delta}$ is contained in only the same type of tiles, otherwise $\Tilde{\delta}$ will have a non-transversal intersection with a lift of a systolic geodesic. Furthermore, $\Tilde{\delta}$ passes through at most two octagons, as $\delta$ is simple. The arcs $\{ \gamma_1,\gamma_2\}$ and $\{\beta_1,\beta_2\}$ in Figure \ref{fig:nonsystole} project to two geodesics satisfying the properties of the lemma. Using the isometries of the Bolza surface and interchanging the octagons, we get the remaining $14$ geodesics.
        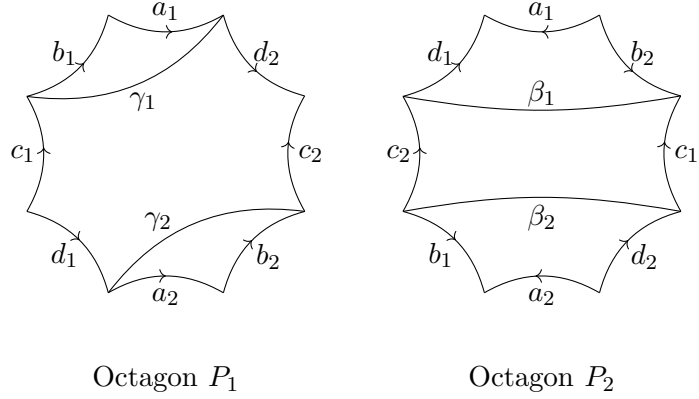
\begin{figure}
            \centering
            \begin{tikzpicture}
                \draw [bend left,-<-=.5] ({2*cos(22.7)},{2*sin(22.7)}) to ({2*cos(67.5)},{2*sin(67.5)});
                \draw [bend left,-<-=.5] ({2*cos(67.5)},{2*sin(67.5)}) to ({2*cos(112.5)},{2*sin(112.5)});
                \draw [bend left,-<-=.5] ({2*cos(112.5)},{2*sin(112.5)}) to ({2*cos(157.5)},{2*sin(157.5)});
                \draw [bend left,-<-=.5] ({2*cos(157.5)},{2*sin(157.5)}) to ({2*cos(202.5)},{2*sin(202.5)});
                \draw [bend left,->-=.5] ({2*cos(202.5)},{2*sin(202.5)}) to ({2*cos(247.5)}, {2*sin(247.5)});
                \draw [bend left,->-=.5] ({2*cos(247.5)}, {2*sin(247.5)}) to ({2*cos(292.5)},{2*sin(292.5)});
                \draw [bend left,->-=.5] ({2*cos(292.5)},{2*sin(292.5)}) to ({2*cos(337.5)},{2*sin(337.5)});
                \draw [bend left,->-=.6] ({2*cos(337.5)},{2*sin(337.5)}) to ({2*cos(22.7)},{2*sin(22.7)});

                \draw [bend left,-<-=.5] ({5+2*cos(22.7)},{2*sin(22.7)}) to ({5+2*cos(67.5)},{2*sin(67.5)});
                \draw [bend left,->-=.5] ({5+2*cos(67.5)},{2*sin(67.5)}) to ({5+2*cos(112.5)},{2*sin(112.5)});
                \draw [bend left,-<-=.5] ({5+2*cos(112.5)},{2*sin(112.5)}) to ({5+2*cos(157.5)},{2*sin(157.5)});
                \draw [bend left,-<-=.5] ({5+2*cos(157.5)},{2*sin(157.5)}) to ({5+2*cos(202.5)},{2*sin(202.5)});
                \draw [bend left,->-=.5] ({5+2*cos(202.5)},{2*sin(202.5)}) to ({5+2*cos(247.5)}, {2*sin(247.5)});
                \draw [bend left,-<-=.5] ({5+2*cos(247.5)}, {2*sin(247.5)}) to ({5+2*cos(292.5)},{2*sin(292.5)});
                \draw [bend left,->-=.5] ({5+2*cos(292.5)},{2*sin(292.5)}) to ({5+2*cos(337.5)},{2*sin(337.5)});
                \draw [bend left,->-=.6] ({5+2*cos(337.5)},{2*sin(337.5)}) to ({5+2*cos(22.7)},{2*sin(22.7)});

                \draw ({1.93*cos(0)},{1.93*sin(0)}) node {$c_2$} ({1.9*cos(45)},{1.9*sin(45)}) node {$d_2$} ({1.9*cos(90)},{1.9*sin(90)}) node {$a_1$} ({1.9*cos(135)},{1.9*sin(135)}) node {$b_1$} ({1.9*cos(180)},{1.9*sin(180)}) node {$c_1$} ({1.9*cos(225)},{1.9*sin(225)}) node {$d_1$} ({1.9*cos(270)},{1.9*sin(270)}) node {$a_2$} ({1.93*cos(315)},{1.93*sin(315)}) node {$b_2$};

                \draw ({5+1.93*cos(0)},{1.93*sin(0)}) node {$c_1$} ({5+1.9*cos(45)},{1.9*sin(45)}) node {$b_2$} ({5+1.9*cos(90)},{1.9*sin(90)}) node {$a_1$} ({5+1.9*cos(135)},{1.9*sin(135)}) node {$d_1$} ({5+1.9*cos(180)},{1.9*sin(180)}) node {$c_2$} ({5+1.9*cos(225)},{1.9*sin(225)}) node {$b_1$} ({5+1.9*cos(270)},{1.9*sin(270)}) node {$a_2$} ({5+1.93*cos(315)},{1.93*sin(315)}) node {$d_2$};

                \draw[bend left] ({2*cos(67.5)},{2*sin(67.5)}) to  ({2*cos(157.5)},{2*sin(157.5)});
                \draw[bend left] ({2*cos(247.5)}, {2*sin(247.5)}) to ({2*cos(337.5)},{2*sin(337.5)});
                \draw (-.3,.7) node{$\gamma_1$};
                \draw (-.1,-.9) node{$\gamma_2$};

                \draw[bend left=10] ({5+2*cos(22.7)},{2*sin(22.7)}) to ({5+2*cos(157.5)},{2*sin(157.5)});
                \draw[bend left=10] ({5+2*cos(202.5)},{2*sin(202.5)}) to ({5+2*cos(337.5)},{2*sin(337.5)});
                \draw (5,.8) node{$\beta_1$} (5,-.85) node{$\beta_2$};
                \draw (0,-3) node {Octagon $P_1$} (5,-3) node{Octagon $P_2$};
            \end{tikzpicture}
            \caption{Non-systolic geodesics on Bolza surface.}
            \label{fig:nonsystole}
        \end{figure}
    \end{proof}
    
   Before proving the next lemma, we introduce the notion of the intersection number of a system of curves with a simple closed curve, which will play a key role throughout the rest of the article. 

\begin{definition}
Let $\Omega$ be a system of curves on a surface, and let $\beta$ be a simple closed curve. We define the intersection number of $\Omega$ with $\beta$ by
$$i(\Omega,\beta)\coloneqq \sum_{\alpha\in\Omega} i(\alpha,\beta).$$
\end{definition}

    As $i(\Omega_1,\delta)\neq0$, there exists a path lift of $\delta$ in the universal cover $\mathbb{D}$, with the initial point on the boundary of a tile in the tessellation described above. Furthermore, we can choose the initial point in the interior of some side in $p^{-1}(\Omega_1)$. Let $\Tilde{\delta}$ be such a path lift of $\delta$. 

    \begin{lemma}\label{lem3.1}
    $i(\Omega_1,\delta)=|p^{-1}(\Omega_1)\cap \Tilde{\delta}|-1$.
    \end{lemma}
    
    \begin{proof}
    The curves considered here are simple closed geodesics and hence they are in a pairwise minimal position. That the initial and the terminal points of $\Tilde{\delta}$ are identified implies  $i(\Omega_1,\delta)\leq|\Gamma\cap \Tilde{\delta}|-1$.  If some interior points of $\Tilde{\delta}$ are identified, then it contradicts that $\delta$ is simple. This completes the proof of the lemma.
    \end{proof}

    \begin{lemma}\label{Thm4.2}
    If $\gamma$ is a non-systolic geodesic in the Bolza surface $S$, then $i(\Omega_1,\gamma)$ is an even integer.
    \end{lemma}
    
    \begin{proof}
      If $\gamma\in \Gamma$ ($\Gamma$ is defined in Lemma \ref{lemma:16geodesics}), then $i(\Omega_1, \gamma)=4$. Now, suppose that $\gamma\notin \Gamma$. As $\Omega_1$ is a filling system, $i(\Omega_1,\gamma)\neq0$. Let $\Tilde{\gamma}$ be a path lift of $\gamma$ with the initial point in the interior of some side of an octagon determined by $p^{-1}(\Omega_1)$. Now, there are the following two cases to be considered.

    \noindent Case 1. If $\Tilde{\gamma}$ does not pass through any of the vertices of the octagons then $\Tilde{\gamma}$ alternatively passes through octagons of Type I and Type II. As $\gamma$ is simple and the sides of Type I octagons are identified with the sides of Type II octagons, if $\Tilde{\gamma}$ starts from a point on a side of a Type I octagon then it ends at a point on a side of a Type II octagon and vice versa. Therefore, $\Tilde{\gamma}$ passes through an even number of octagons, say $2n$. As $\Tilde{\gamma}$ intersects each of them twice, by Lemma \ref{lem3.1}, we have $i(\Omega_1,\gamma)=2n$, which is even.

    \noindent Case 2. If $\Tilde{\gamma}$ passes through some vertices, then we perturb $\Tilde{\gamma}$ near the vertices and get a new arc $\Tilde{\gamma}'$ which does not pass through a vertex (see Figure \ref{fig1.2}). Then $i(\Omega_1,\gamma')= i(\Omega_1,\gamma)$, where $\gamma'$ is the image of $\Tilde{\gamma}'$ in the surface and $\Tilde{\gamma}'$ alternatively passes the octagons of Type I and Type II. A similar argument as in Case 1 implies that $i(\Omega_1, \gamma)$ is even.
\begin{figure}[htbp]
        \centering
        \begin{tikzpicture}
            \draw (-1,1)--(1,-1);
            \draw (1,1)--(-1,-1);
            \draw[red] (-1,0)--(1,0);
            
            \draw (3.5,1)--(5.5,-1);
            \draw (5.5,1)--(3.5,-1);
            \draw[red] (3.5,0)--(4,0);
            \draw[red] (4,0) [bend left=50] to (5,0);
            \draw[red] (5,0)--(5.5,0);
            
            \draw (-1.2,0) node{$\Tilde{\gamma}$};
            \draw (3.2,0) node{$\Tilde{\gamma}'$};

            \draw (2,0) node {$\longrightarrow$};        \end{tikzpicture}
        \caption{Perturbation of $\Tilde{\gamma}$.}
        \label{fig1.2}
    \end{figure}
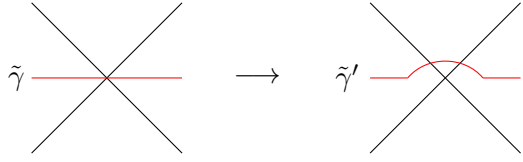
    \end{proof}
    
    \begin{lemma}\label{lemma 4.4}
    Let $\Tilde{\gamma}$ be an arc in $\mathbb D$ without self-intersection such that the end points of $\Tilde{\gamma}$ are in the same orbit and $\Tilde{\gamma}$ contains no vertices of $p^{-1}(\Omega_1)$. Then, $\Tilde{\gamma}$ intersects $p^{-1}(\Omega_2)$ at even number of points.
    \end{lemma}
    
    \begin{proof}
    \begin{figure}
    \centering
    \begin{tikzpicture}
        \draw [bend left,-<-=.5] ({2*cos(22.7)},{2*sin(22.7)}) to ({2*cos(67.5)},{2*sin(67.5)});
         \draw [bend left,-<-=.5] ({2*cos(67.5)},{2*sin(67.5)}) to ({2*cos(112.5)},{2*sin(112.5)});
         \draw [bend left,-<-=.5] ({2*cos(112.5)},{2*sin(112.5)}) to ({2*cos(157.5)},{2*sin(157.5)});
          \draw [bend left,-<-=.47] ({2*cos(157.5)},{2*sin(157.5)}) to ({2*cos(202.5)},{2*sin(202.5)});
          \draw [bend left,->-=.5] ({2*cos(202.5)},{2*sin(202.5)}) to ({2*cos(247.5)}, {2*sin(247.5)});
          \draw [bend left,->-=.5] ({2*cos(247.5)}, {2*sin(247.5)}) to ({2*cos(292.5)},{2*sin(292.5)});
          \draw [bend left,->-=.5] ({2*cos(292.5)},{2*sin(292.5)}) to ({2*cos(337.5)},{2*sin(337.5)});
         \draw [dotted,thick] ({2*cos(337.5)},{2*sin(337.5)}) to ({2*cos(22.7)},{2*sin(22.7)});

          \draw [bend left,-<-=.5] ({4*cos(22.7)+2*cos(22.7)},{2*sin(22.7)}) to ({4*cos(22.7)+2*cos(67.5)},{2*sin(67.5)});
          \draw [bend left,->-=.5] ({4*cos(22.7)+2*cos(67.5)},{2*sin(67.5)}) to ({4*cos(22.7)+2*cos(112.5)},{2*sin(112.5)});
          \draw [bend left,-<-=.5] ({4*cos(22.7)+2*cos(112.5)},{2*sin(112.5)}) to ({4*cos(22.7)+2*cos(157.5)},{2*sin(157.5)});
          \draw [bend left,->-=.5] ({4*cos(22.7)+2*cos(202.5)},{2*sin(202.5)}) to ({4*cos(22.7)+2*cos(247.5)}, {2*sin(247.5)});
          \draw [bend left,-<-=.5] ({4*cos(22.7)+2*cos(247.5)}, {2*sin(247.5)}) to ({4*cos(22.7)+2*cos(292.5)},{2*sin(292.5)});
          \draw [bend left,->-=.5] ({4*cos(22.7)+2*cos(292.5)},{2*sin(292.5)}) to ({4*cos(22.7)+2*cos(337.5)},{2*sin(337.5)});
         \draw [bend left,->-=.6] ({4*cos(22.7)+2*cos(337.5)},{2*sin(337.5)}) to ({4*cos(22.7)+2*cos(22.7)},{2*sin(22.7)});

         \draw[thick,red] (-1.61,0) to ({4*cos(22.7)+1.61},0);

         \draw  ({1.9*cos(45)},{1.9*sin(45)}) node {$d_2$} ({1.9*cos(90)},{1.9*sin(90)}) node {$a_1$} ({1.9*cos(135)},{1.9*sin(135)}) node {$b_1$} ({1.9*cos(180)},{1.9*sin(180)}) node {$c_1$} ({1.9*cos(225)},{1.9*sin(225)}) node {$d_1$} ({1.9*cos(270)},{1.9*sin(270)}) node {$a_2$} ({1.93*cos(315)},{1.93*sin(315)}) node {$b_2$};

         \draw ({4*cos(22.7)+1.93*cos(0)},{1.93*sin(0)}) node {$c_1$} ({4*cos(22.7)+1.9*cos(45)},{1.9*sin(45)}) node {$b_2$} ({4*cos(22.7)+1.9*cos(90)},{1.9*sin(90)}) node {$a_1$} ({4*cos(22.7)+1.9*cos(135)},{1.9*sin(135)}) node {$d_1$}  ({4*cos(22.7)+1.9*cos(225)},{1.9*sin(225)}) node {$b_1$} ({4*cos(22.7)+1.9*cos(270)},{1.9*sin(270)}) node {$a_2$} ({4*cos(22.7)+1.93*cos(315)},{1.93*sin(315)}) node {$d_2$};

         \draw (.5,-.25) node{$\delta$};
    \end{tikzpicture}
    \caption{}
    \label{fig:3.2}
    \end{figure}

    The arc $\Tilde{\gamma}$ alternatively passes through octagons of Type I and Type II, as $\Tilde{\gamma}$ does not pass through the vertices. Since the initial point and the end point are identified, $\Tilde{\gamma}$ passes through an even number of octagons, say $2n$. Now, we use mathematical induction on $n$ to complete the proof of the lemma. If $n=1$, then $p(\Tilde{\gamma})$ is either homotopic to a systolic geodesic $\alpha\in \Omega_1$ or homotopic to a geodesic $\delta$, as shown in Figure \ref{fig:3.2}. We have $|\alpha\cap \Omega_2|=4$ and $|\delta\cap \Omega_2|=8$. Therefore, it follows that $\Tilde{\gamma}$ intersects $p^{-1}(\Omega_2)$ at even number of points.

    Next, suppose that the lemma holds true for $n$, where $n\geq1$. Let an arc $\Tilde{\gamma}$ pass through $2(n+1)$ octagons $P_1,\dots,P_{2n+2}$. In order to use the induction hypothesis, we homotope the arc $\Tilde{\gamma}$ to another arc $\Tilde{\gamma}'$ which is the union of the arcs $\delta_i, i=1,\dots ,4$ (see Figure \ref{fig:4.4}). We choose the terminal points of $\delta_1$ in $P_{2n}$ such that both the endpoints of $\delta_1$ are identified. Then, by induction hypothesis $|\delta_1\cap p^{-1}(\Omega_2)|$ is even; the case $n=1$ implies  $|\delta_3\cap p^{-1}(\Omega_2)|$ is even and $|\delta_2\cap p^{-1}(\Omega_2)|=|\delta_4\cap p^{-1}(\Omega_2)|$. Hence, $\Tilde{\gamma}'$ intersects $p^{-1}(\Omega_2)$ at an even number of points and so does $\Tilde{\gamma}$.
    \begin{figure}[htbp]
        \centering
        \begin{tikzpicture}
            \draw [-<-=.5] ({cos(22.7)},{sin(22.7)}) to ({cos(67.5)},{sin(67.5)});
         \draw [-<-=.5] ({cos(67.5)},{sin(67.5)}) to ({cos(112.5)},{sin(112.5)});
         \draw [-<-=.5] ({cos(112.5)},{sin(112.5)}) to ({cos(157.5)},{sin(157.5)});
          \draw [-<-=.5] ({cos(157.5)},{sin(157.5)}) to ({cos(202.5)},{sin(202.5)});
          \draw [->-=.5] ({cos(202.5)},{sin(202.5)}) to ({cos(247.5)}, {sin(247.5)});
          \draw [->-=.5] ({cos(247.5)}, {sin(247.5)}) to ({cos(292.5)},{sin(292.5)});
          \draw [->-=.5] ({cos(292.5)},{sin(292.5)}) to ({cos(337.5)},{sin(337.5)});

         \draw ({1.13*cos(8)},{1.13*sin(8)}) node{\scriptsize$c_2$} ({1.12*cos(50)},{1.12*sin(50)}) node{\scriptsize$d_2$} ({1.1*cos(90)},{1.1*sin(90)}) node{\scriptsize$a_1$} ({1.18*cos(135)},{1.18*sin(135)}) node{\scriptsize$b_1$} ({1.16*cos(180)},{1.16*sin(180)}) node{\scriptsize$c_1$} ({1.14*cos(225)},{1.14*sin(225)}) node{\scriptsize$d_1$} ({1.1*cos(280)},{1.1*sin(280)}) node{\scriptsize$a_2$} ({1.18*cos(310)},{1.18*sin(310)}) node{\scriptsize$b_2$} ;

         \draw [-<-=.45] ({2*cos(22.7)+cos(22.7)},{sin(22.7)}) to ({2*cos(22.7)+cos(67.5)},{sin(67.5)});
         \draw [->-=.5] ({2*cos(22.7)+cos(67.5)},{sin(67.5)}) to ({2*cos(22.7)+cos(112.5)},{sin(112.5)});
         \draw [-<-=.5] ({2*cos(22.7)+cos(112.5)},{sin(112.5)}) to ({2*cos(22.7)+cos(157.5)},{sin(157.5)});
          \draw [-<-=.5] ({2*cos(22.7)+cos(157.5)},{sin(157.5)}) to ({2*cos(22.7)+cos(202.5)},{sin(202.5)});
          \draw [->-=.5] ({2*cos(22.7)+cos(202.5)},{sin(202.5)}) to ({2*cos(22.7)+cos(247.5)}, {sin(247.5)});
          \draw [-<-=.5] ({2*cos(22.7)+cos(247.5)}, {sin(247.5)}) to ({2*cos(22.7)+cos(292.5)},{sin(292.5)});
          \draw [->-=.5] ({2*cos(22.7)+cos(292.5)},{sin(292.5)}) to ({2*cos(22.7)+cos(337.5)},{sin(337.5)});
         \draw [->-=.6] ({2*cos(22.7)+cos(337.5)},{sin(337.5)}) to ({2*cos(22.7)+cos(22.7)},{sin(22.7)});

         \draw ({2*cos(22.7)+1.2*cos(0)},{1.2*sin(0)}) node {\scriptsize$c_1$} ({2*cos(22.7)+1.1*cos(50)},{1.1*sin(50)}) node {\scriptsize$b_2$} ({2*cos(22.7)+1.15*cos(90)},{1.15*sin(90)}) node {\scriptsize$a_1$} ({2*cos(22.7)+1.12*cos(130)},{1.12*sin(130)}) node {\scriptsize$d_1$} ({2*cos(22.7)+1.1*cos(230)},{1.1*sin(230)}) node {\scriptsize$b_1$} ({2*cos(22.7)+1.1*cos(280)},{1.1*sin(280)}) node {\scriptsize$a_2$} ({2*cos(22.7)+1.14*cos(315)},{1.14*sin(315)}) node {\scriptsize$d_2$};

          \draw [->-=.5] ({5+cos(22.7)},{1.5+sin(22.7)}) to ({5+cos(67.5)},{1.5+sin(67.5)});
         \draw [-<-=.3] ({5+cos(67.5)},{1.5+sin(67.5)}) to ({5+cos(112.5)},{1.5+sin(112.5)});
         \draw [-<-=.5] ({5+cos(112.5)},{1.5+sin(112.5)}) to ({5+cos(157.5)},{1.5+sin(157.5)});
          \draw [->-=.5] ({5+cos(157.5)},{1.5+sin(157.5)}) to ({5+cos(202.5)},{1.5+sin(202.5)});
          \draw [-<-=.5] ({5+cos(202.5)},{1.5+sin(202.5)}) to ({5+cos(247.5)}, {1.5+sin(247.5)});
          \draw [->-=.5] ({5+cos(247.5)}, {1.5+sin(247.5)}) to ({5+cos(292.5)},{1.5+sin(292.5)});
          \draw [->-=.5] ({5+cos(292.5)},{1.5+sin(292.5)}) to ({5+cos(337.5)},{1.5+sin(337.5)});

         \draw ({5+1.16*cos(350)},{1.5+1.16*sin(350)}) node {\scriptsize$b_2$} ({5+1.1*cos(50)},{1.5+1.1*sin(50)}) node {\scriptsize$a_1$} ({5+1.1*cos(90)},{1.5+1.1*sin(90)}) node {\scriptsize$d_1$} ({5+1.16*cos(135)},{1.5+1.16*sin(135)}) node {\scriptsize$c_2$} ({5+1.16*cos(180)},{1.5+1.16*sin(180)}) node {\scriptsize$b_1$} ({5+1.16*cos(215)},{1.5+1.16*sin(210)}) node {\scriptsize$a_2$} ({5+1.1*cos(280)},{1.5+1.1*sin(280)}) node {\scriptsize$d_2$} ({5+1.16*cos(310)},{1.5+1.16*sin(310)}) node {\scriptsize$c_1$};

         \draw [->-=.5] ({5+2*cos(22.7)+cos(67.5)},{1.5+sin(67.5)}) to ({5+2*cos(22.7)+cos(112.5)},{1.5+sin(112.5)});
         \draw [->-=.5] ({5+2*cos(22.7)+cos(112.5)},{1.5+sin(112.5)}) to ({5+2*cos(22.7)+cos(157.5)},{1.5+sin(157.5)});
          \draw [->-=.45] ({5+2*cos(22.7)+cos(157.5)},{1.5+sin(157.5)}) to ({5+2*cos(22.7)+cos(202.5)},{1.5+sin(202.5)});
          \draw [->-=.5] ({5+2*cos(22.7)+cos(202.5)},{1.5+sin(202.5)}) to ({5+2*cos(22.7)+cos(247.5)}, {1.5+sin(247.5)});
          \draw [-<-=.5] ({5+2*cos(22.7)+cos(247.5)}, {1.5+sin(247.5)}) to ({5+2*cos(22.7)+cos(292.5)},{1.5+sin(292.5)});
          \draw [-<-=.5] ({5+2*cos(22.7)+cos(292.5)},{1.5+sin(292.5)}) to ({5+2*cos(22.7)+cos(337.5)},{1.5+sin(337.5)});
         \draw [-<-=.6] ({5+2*cos(22.7)+cos(337.5)},{1.5+sin(337.5)}) to ({5+2*cos(22.7)+cos(22.7)},{1.5+sin(22.7)});

         \draw ({5+2*cos(22.7)+1.16*cos(355)},{1.5+1.16*sin(355)}) node {\scriptsize $b_1$}  ({5+2*cos(22.7)+1.12*cos(50)},{1.5+1.12*sin(50)}) node {\scriptsize $c_1$} ({5+2*cos(22.7)+1.16*cos(95)},{1.5+1.16*sin(95)}) node {\scriptsize $d_1$} ({5+2*cos(22.7)+1.16*cos(130)},{1.5+1.16*sin(130)}) node {\scriptsize $a_2$} ({5+2*cos(22.7)+1.14*cos(230)},{1.5+1.14*sin(230)}) node {\scriptsize $c_2$} ({5+2*cos(22.7)+1.12*cos(280)},{1.5+1.12*sin(280)}) node {\scriptsize $d_2$} ({5+2*cos(22.7)+1.16*cos(315)},{1.5+1.16*sin(315)}) node {\scriptsize $a_1$};

         \draw [-<-=.5] ({2*cos(22.7)*cos(45)+5+2*cos(22.7)+cos(22.7)},{2*cos(22.7)*sin(45)+1.5+sin(22.7)}) to ({2*cos(22.7)*cos(45)+5+2*cos(22.7)+cos(67.5)},{2*cos(22.7)*sin(45)+1.5+sin(67.5)});
         \draw [->-=.5] ({2*cos(22.7)*cos(45)+5+2*cos(22.7)+cos(67.5)},{2*cos(22.7)*sin(45)+1.5+sin(67.5)}) to ({2*cos(22.7)*cos(45)+5+2*cos(22.7)+cos(112.5)},{2*cos(22.7)*sin(45)+1.5+sin(112.5)});
         \draw [-<-=.5] ({2*cos(22.7)*cos(45)+5+2*cos(22.7)+cos(112.5)},{2*cos(22.7)*sin(45)+1.5+sin(112.5)}) to ({2*cos(22.7)*cos(45)+5+2*cos(22.7)+cos(157.5)},{2*cos(22.7)*sin(45)+1.5+sin(157.5)});
          \draw [->-=.5] ({2*cos(22.7)*cos(45)+5+2*cos(22.7)+cos(157.5)},{2*cos(22.7)*sin(45)+1.5+sin(157.5)}) to ({2*cos(22.7)*cos(45)+5+2*cos(22.7)+cos(202.5)},{2*cos(22.7)*sin(45)+1.5+sin(202.5)});
          \draw [->-=.4] ({2*cos(22.7)*cos(45)+5+2*cos(22.7)+cos(202.5)},{2*cos(22.7)*sin(45)+1.5+sin(202.5)}) to ({2*cos(22.7)*cos(45)+5+2*cos(22.7)+cos(247.5)}, {2*cos(22.7)*sin(45)+1.5+sin(247.5)});
          \draw [-<-=.35] ({2*cos(22.7)*cos(45)+5+2*cos(22.7)+cos(247.5)}, {2*cos(22.7)*sin(45)+1.5+sin(247.5)}) to ({2*cos(22.7)*cos(45)+5+2*cos(22.7)+cos(292.5)},{2*cos(22.7)*sin(45)+1.5+sin(292.5)});
          \draw [->-=.5] ({2*cos(22.7)*cos(45)+5+2*cos(22.7)+cos(292.5)},{2*cos(22.7)*sin(45)+1.5+sin(292.5)}) to ({2*cos(22.7)*cos(45)+5+2*cos(22.7)+cos(337.5)},{2*cos(22.7)*sin(45)+1.5+sin(337.5)});
         \draw [-<-=.4] ({2*cos(22.7)*cos(45)+5+2*cos(22.7)+cos(337.5)},{2*cos(22.7)*sin(45)+1.5+sin(337.5)}) to ({2*cos(22.7)*cos(45)+5+2*cos(22.7)+cos(22.7)},{2*cos(22.7)*sin(45)+1.5+sin(22.7)});

         \draw ({2*cos(22.7)*cos(45)+5+2*cos(22.7)+1.1*cos(275)},{2*cos(22.7)*sin(45)+1.5+1.1*sin(275)}) node{\scriptsize$b_2$} ({2*cos(22.7)*cos(45)+5+2*cos(22.7)+1.16*cos(320)},{2*cos(22.7)*sin(45)+1.5+1.16*sin(320)}) node{\scriptsize$a_1$} ({2*cos(22.7)*cos(45)+5+2*cos(22.7)+1.2*cos(0)},{2*cos(22.7)*sin(45)+1.5+1.2*sin(0)}) node{\scriptsize$d_1$} ({2*cos(22.7)*cos(45)+5+2*cos(22.7)+1.1*cos(45)},{2*cos(22.7)*sin(45)+1.5+1.1*sin(45)}) node{\scriptsize$c_2$} ({2*cos(22.7)*cos(45)+5+2*cos(22.7)+1.2*cos(90)},{2*cos(22.7)*sin(45)+1.5+1.2*sin(90)}) node{\scriptsize$b_1$} ({2*cos(22.7)*cos(45)+5+2*cos(22.7)+1.18*cos(135)},{2*cos(22.7)*sin(45)+1.5+1.18*sin(135)}) node{\scriptsize$a_2$} ({2*cos(22.7)*cos(45)+5+2*cos(22.7)+1.16*cos(170)},{2*cos(22.7)*sin(45)+1.5+1.16*sin(170)}) node{\scriptsize$d_2$};

         \draw (0,-1.5) node {$P_1$};
         \draw ({2*cos(22.7)},-1.5) node {$P_2$};
         \draw (5,0) node {$P_{2n}$};
         \draw ({5+2*cos(22.7)},{1.5-1.5}) node {$P_{2n+1}$};
         \draw (9.3,1.7) node {$P_{2n+2}$};

         \draw ({.95*cos(230)},{.95*sin(230)})..controls(0,-.2) and ({2*cos(22.7)},-.5)..({2*cos(22.7)+.95*cos(35)},{.95*sin(35)});
         \draw[dashed] ({2*cos(22.7)+.95*cos(35)},{.95*sin(35)})..controls (3.2,1.1)and(4.2,.3)..({5+.95*cos(240)},{1.5+.95*sin(240)});
         \draw ({5+.95*cos(240)},{1.5+.95*sin(240)})..controls(5,1.5)and(7,1.4).. ({2*cos(22.7)*cos(45)+5+2*cos(22.7)+.95*cos(10)},{2*cos(22.7)*sin(45)+1.5+.95*sin(10)});

         \draw[ForestGreen] ({.95*cos(230)},{.95*sin(230)})..controls(0,0.2) and ({2*cos(22.7)},-.3)..({2*cos(22.7)+.95*cos(45)},{.95*sin(45)});
         \draw[dashed,ForestGreen] ({2*cos(22.7)+.95*cos(45)},{.95*sin(45)})..controls (3.2,1.3)and(4.2,.5)..({5+.95*cos(230)},{1.5+.95*sin(230)});
         \draw[ForestGreen] ({5+.95*cos(230)},{1.5+.95*sin(230)})..controls(5.1,1.2)and(4,2)..({5+.92*cos(95)},{1.5+.92*sin(95)});
         \draw[red] ({5+.92*cos(95)},{1.5+.92*sin(95)})..controls (5.3,2.5)and(5.6,1.6)..({5+.93*cos(10)},{1.5+.93*sin(10)});
         \draw[blue] ({5+.93*cos(10)},{1.5+.93*sin(10)})..controls(6.8,1.8)and(7,2.5)..({2*cos(22.7)*cos(45)+5+2*cos(22.7)+.93*cos(280)},{2*cos(22.7)*sin(45)+1.5+.93*sin(280)});
         \draw[red] ({2*cos(22.7)*cos(45)+5+2*cos(22.7)+.93*cos(280)},{2*cos(22.7)*sin(45)+1.5+.93*sin(280)})..controls(8.4,1.88) and (8.45,1.95)..({2*cos(22.7)*cos(45)+5+2*cos(22.7)+.95*cos(10)},{2*cos(22.7)*sin(45)+1.5+.95*sin(10)});

         \draw (0,-.05) node{\scriptsize$\delta_1$} (5.4,1.9) node{\scriptsize$\delta_2$} (6.6,2.1) node{\scriptsize$\delta_3$} (8.42,2.25) node{\scriptsize$\delta_4$} (1.8,-.28) node{\scriptsize$\Tilde{\gamma}$};
        \end{tikzpicture}
        \caption{Deformation of the arc $\Tilde{\gamma}$}
        \label{fig:4.4}
    \end{figure}
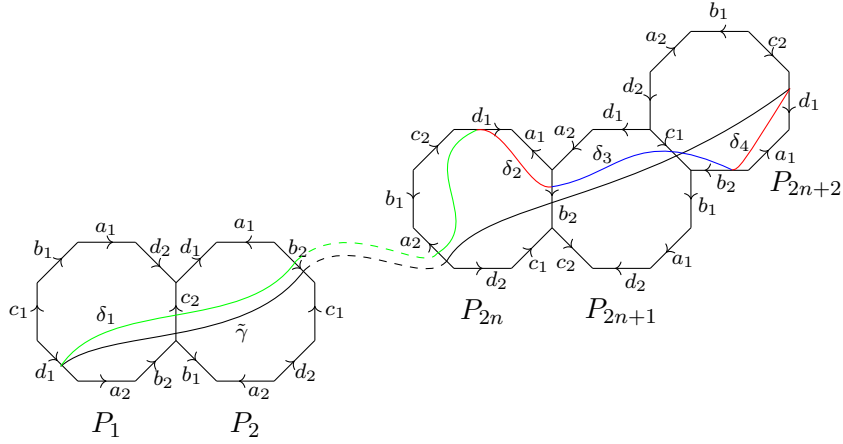
\end{proof}

Now we are ready to prove Theorem \ref{thm:3.1}.

\begin{proof}[Proof of Theorem \ref{thm:3.1}]
    Let $\gamma$ be a non-systolic simple closed geodesic on the Bolza surface $S$. If $\gamma\in\Gamma$, then from  Lemma \ref{lemma:16geodesics}, it follows that $i(\Sys(S),\gamma )=10 \text{ or } 12$. In the remaining cases, we may assume that $\gamma$ does not pass through any of the points of intersection of $\Sys (S)$, otherwise, we deform the curve as in Figure \ref{fig1.2} and this deformation does not change the number of points of intersection. As $\Omega_1$ fills the surface $S$, the geodesic $\gamma$ intersects at least one geodesic of $\Omega_1$. Let $\Tilde{\gamma}$ be a path lift of $\gamma$ with initial point on $p^{-1}(\Omega_1)$. Then, $\Tilde{\gamma}$ satisfies all the hypothesis of Lemma \ref{lemma 4.4} and hence $\Tilde{\gamma}$ intersects $p^{-1}(\Omega_2)$ at an even number of points. This implies that $i(\Omega_2,\gamma)$ is even. The rest of the proof follows from Lemma \ref{Thm4.2}.
\end{proof}

\subsection{Even Complexity} In this subsection, we prove that the $k$-th complexity $\mathcal{G}_k(\Sys(S))$ of $\Sys(S)$ is nonzero for infinitely many even integers $k$. In particular, we prove Theorem \ref{thm:3.8}. Before proving this theorem, we establish a simple but interesting result that gives a necessary and sufficient condition for the complexity of a collection of curves to be finite. 

\begin{proposition}\label{prop:3.6}
    Let $\Sigma$ be a curve system on a hyperbolic surface $X$. Then the $k$-th complexity $\mathcal{G}_k(\Sigma)$ is finite if and only if $\Sigma$ is a filling.
\end{proposition}
\begin{proof}
    ($\Leftarrow$) Suppose $\Sigma$ is a filling system in $X$. Let $X\approx\Ha/G$, where $G$ is a Fuchsian group and $\Ha$ is the upper half plane. Let $F$ be a polygonal fundamental domain for $G$ determined by the lifts of $\Sigma$. Let $\delta$ be a simple closed curve on $X$ with $i(\Sigma, \delta)=k$ and $\Tilde{\delta}$ be a path lift of $\delta$ with the initial point on the boundary of $F$. Then $\Tilde{\delta}$ passes through at most $k$ copies of $F$ and this implies that $l_X(\gamma)\leq k\cdot d_F$, where $d_F$ is the diameter of $F$. As there are only finitely many simple closed geodesics in a hyperbolic surface with bounded length (Lemma 12.4 \cite{FarbMargalit}), the result follows.

    ($\Rightarrow$) If $\Sigma$ is not a filling, then there exists a simple closed curve $\beta$ such that $i(\Sigma,\beta)=0$. Now, we consider a simple closed curve $\alpha_0$ with $i(\alpha_0,\beta)\neq0$ and the Dehn twist $T_\beta$ about the curve $\beta$. Then the curves $\alpha_n=T^n_{\beta}(\alpha_0), n\in \Z,$ are all distinct and satisfies $i(\Sigma,\alpha_0)=i(\Sigma,\alpha_n)$, for all $n\in \Z$. This implies that  $\mathcal{G}_{k_0}(\Sigma)$ is infinite, where $k_0=i(\Sigma,\alpha_0)$.
\end{proof}

\begin{corollary}
    $\mathcal{G}_k(\Sys(S))$ is finite for all $k\in\Z_{\geq0}$.
\end{corollary}
\begin{proof}
    The set $\Sys(S)$ is a filling as it decomposes the surface into triangles (see \cite{MR3877282}). The proof now follows.
\end{proof}

\textbf{An estimate of $\mathcal{G}_k(\Sys(S))$.} Let $\gamma$ be a non-systolic geodesic on $S$ with $i(\Sys (S), \gamma)\leq k$. Then a path lift $\Tilde{\gamma}$ of $\gamma$ passes through at most $k/2$ octagons. So, the length of $\gamma$ is bounded above by $k/2$ times the diameter of an octagon. That is, $l(\gamma)\leq k\cdot \cosh^{-1}\left(1+\sqrt{2}\right)$. By using Theorem 4.1~\cite{Rivin}  due to I. Rivin, we have
$$\mathcal{G}_k(\Sys(S))\leq c\left[k\cdot \cosh^{-1}\left(1+\sqrt{2}\right)\right]^6,$$
where $c$ is a constant.

Now we prove the main theorem of this subsection.

\mainthmtwo*
\begin{proof}
    Here, we explicitly find curves that intersect $\Sys(S)$ at $(10+6n)$ or $(12+6n)$ times, where $n$ is a non-negative integer. Consider the geodesic $\gamma=\gamma_1\cup\gamma_2$ as in Figure \ref{fig:nonsystole}. Then, $i(\Sys(S), \gamma)=10$ and $\gamma$ intersects the systolic geodesic $a=a_1 \cup a_2$ exactly once. Now, consider the family of curves $T^n_a(\gamma)$, where $T_a$ is the Dehn twist around the curve $a$ and $n$ is a non-negative integer. As there are $6$ other systolic geodesics, each of which intersects $a$ exactly once, we have 
    $$i(\Sys(S),T^n_a(\gamma))=10+6n.$$
    Next, consider the geodesic $\beta=\beta_1\cup\beta_2$ (see Figure \ref{fig:nonsystole}). Then $i(\Sys(S),\beta)=12$ and $\beta$ intersects the systolic geodesic $b=b_1 \cup b_2$ exactly once. Now, as in the earlier case 
    $$i(\Sys(S),T^n_b(\beta))=12+6n.$$
\end{proof}

\section{Second systoles}
In this section, we prove the theorem below.
\mainthmthree*
\begin{proof}
    Aurich--Steiner~\cite{AurStein} have shown that the length $l_2$ of the second systolic geodesic of the Bolza surface $S$ is given by $\cosh(l_2/2)=3+2\sqrt{2}$. Consider the Bolza surface as given in Figure~\ref{Fif4.1}(i). The diagonals of the octagon joining opposite vertices project onto simple closed geodesics on $S$. Also, the length of these geodesics is $l_2$. Using the isometries on these geodesics, we get $12$ second systolic geodesics (see Figure~\ref{Fif4.1}). Now, we show that these are the only second systolic geodesics on the Bolza surface.

    Let $\alpha$ be a simple closed curve on $S$. If $\alpha$ is separating, then $\alpha$ contains $4$ arcs joining non-consecutive sides of the octagon and hence $l(\alpha)>l_2$. If $\alpha$ is non-separating, then by Theorem~\ref{thm2.4}, $\alpha$ passes through some fixed points of the hyperelliptic involution $\sigma$ and in order to become a second systolic geodesic, $\alpha$ must pass through exactly two fixed points. This implies that the geodesics shown in Figure~\ref{Fif4.1}(ii) are the only $12$ second systolic geodesics on $S$.
\end{proof}


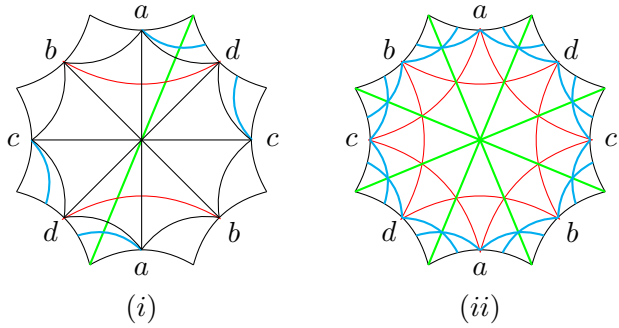
\begin{figure}[htbp]
    \centering
    \begin{tikzpicture}[xscale=.9,yscale=.9]
         \draw [bend left] ({2*cos(22.7)},{2*sin(22.7)}) to ({2*cos(67.5)},{2*sin(67.5)});
         \draw [bend left] ({2*cos(67.5)},{2*sin(67.5)}) to ({2*cos(112.5)},{2*sin(112.5)});
         \draw [bend left] ({2*cos(112.5)},{2*sin(112.5)}) to ({2*cos(157.5)},{2*sin(157.5)});
          \draw [bend left] ({2*cos(157.5)},{2*sin(157.5)}) to ({2*cos(202.5)},{2*sin(202.5)});
          \draw [bend left] ({2*cos(202.5)},{2*sin(202.5)}) to ({2*cos(247.5)}, {2*sin(247.5)});
          \draw [bend left] ({2*cos(247.5)}, {2*sin(247.5)}) to ({2*cos(292.5)},{2*sin(292.5)});
          \draw [bend left] ({2*cos(292.5)},{2*sin(292.5)}) to ({2*cos(337.5)},{2*sin(337.5)});
         \draw [bend left] ({2*cos(337.5)},{2*sin(337.5)}) to ({2*cos(22.7)},{2*sin(22.7)});

         \draw ({1.93*cos(0)},{1.93*sin(0)}) node {$c$} ({1.9*cos(45)},{1.9*sin(45)}) node {$d$} ({1.9*cos(90)},{1.9*sin(90)}) node {$a$} ({1.9*cos(135)},{1.9*sin(135)}) node {$b$} ({1.9*cos(180)},{1.9*sin(180)}) node {$c$} ({1.9*cos(225)},{1.9*sin(225)}) node {$d$} ({1.9*cos(270)},{1.9*sin(270)}) node {$a$} ({1.93*cos(315)},{1.93*sin(315)}) node {$b$};

         \draw[thick,ForestGreen] ({2*cos(67.5)},{2*sin(67.5)})--({2*cos(247.5)}, {2*sin(247.5)});

         \draw[bend left, red] ({1.65*cos(45)},{1.65*sin(45)}) to ({1.65*cos(135)},{1.65*sin(135)});
         \draw[bend left, red] ({1.665*cos(225)},{1.665*sin(225)}) to ({1.65*cos(315)},{1.65*sin(315)});

         \draw [thick,cyan, bend right] ({1.63*cos(90)},{1.63*sin(90)}) to ({1.7*cos(56.25)},{1.7*sin(56.25)});
         \draw [thick,cyan, bend left] ({1.63*cos(180)},{1.63*sin(180)}) to ({1.7*cos(213.75)},{1.7*sin(213.75)});
         \draw [thick,cyan, bend left] ({1.63*cos(0)},{1.63*sin(0)}) to ({1.7*cos(33.75)},{1.7*sin(33.75)});
         \draw [thick,cyan, bend right] ({1.63*cos(270)},{1.63*sin(270)}) to ({1.7*cos(236.25)},{1.7*sin(236.25)});
         
         \draw [bend left] ({5+2*cos(22.7)},{2*sin(22.7)}) to ({5+2*cos(67.5)},{2*sin(67.5)});
         \draw [bend left] ({5+2*cos(67.5)},{2*sin(67.5)}) to ({5+2*cos(112.5)},{2*sin(112.5)});
         \draw [bend left] ({5+2*cos(112.5)},{2*sin(112.5)}) to ({5+2*cos(157.5)},{2*sin(157.5)});
          \draw [bend left] ({5+2*cos(157.5)},{2*sin(157.5)}) to ({5+2*cos(202.5)},{2*sin(202.5)});
          \draw [bend left] ({5+2*cos(202.5)},{2*sin(202.5)}) to ({5+2*cos(247.5)}, {2*sin(247.5)});
          \draw [bend left] ({5+2*cos(247.5)}, {2*sin(247.5)}) to ({5+2*cos(292.5)},{2*sin(292.5)});
          \draw [bend left] ({5+2*cos(292.5)},{2*sin(292.5)}) to ({5+2*cos(337.5)},{2*sin(337.5)});
         \draw [bend left] ({5+2*cos(337.5)},{2*sin(337.5)}) to ({5+2*cos(22.7)},{2*sin(22.7)});

          \draw ({5+1.93*cos(0)},{1.93*sin(0)}) node {$c$} ({5+1.9*cos(45)},{1.9*sin(45)}) node {$d$} ({5+1.9*cos(90)},{1.9*sin(90)}) node {$a$} ({5+1.9*cos(135)},{1.9*sin(135)}) node {$b$} ({5+1.9*cos(180)},{1.9*sin(180)}) node {$c$} ({5+1.9*cos(225)},{1.9*sin(225)}) node {$d$} ({5+1.9*cos(270)},{1.9*sin(270)}) node {$a$} ({5+1.93*cos(315)},{1.93*sin(315)}) node {$b$};

         \draw[bend left, red] ({5+1.65*cos(45)},{1.65*sin(45)}) to ({5+1.65*cos(135)},{1.65*sin(135)});
         \draw[bend left, red] ({5+1.665*cos(225)},{1.665*sin(225)}) to ({5+1.65*cos(315)},{1.65*sin(315)});
         \draw[bend left, red] ({5+1.65*cos(45+45)},{1.65*sin(45+45)}) to ({5+1.65*cos(45+135)},{1.65*sin(45+135)});
         \draw[bend left, red] ({5+1.665*cos(45+225)},{1.665*sin(45+225)}) to ({5+1.65*cos(45+315)},{1.65*sin(45+315)});
         \draw[bend left, red] ({5+1.65*cos(90+45)},{1.65*sin(90+45)}) to ({5+1.65*cos(90+135)},{1.65*sin(90+135)});
         \draw[bend left, red] ({5+1.665*cos(90+225)},{1.665*sin(90+225)}) to ({5+1.65*cos(90+315)},{1.65*sin(90+315)});
         \draw[bend left, red] ({5+1.65*cos(135+45)},{1.65*sin(135+45)}) to ({5+1.65*cos(135+135)},{1.65*sin(135+135)});
         \draw[bend left, red] ({5+1.665*cos(135+225)},{1.665*sin(135+225)}) to ({5+1.65*cos(135+315)},{1.65*sin(135+315)});

         \draw[thick,ForestGreen] ({5+2*cos(67.5)},{2*sin(67.5)})--({5+2*cos(247.5)}, {2*sin(247.5)});
         \draw[thick,ForestGreen] ({5+2*cos(45+67.5)},{2*sin(45+67.5)})--({5+2*cos(45+247.5)}, {2*sin(45+247.5)});
         \draw[thick,ForestGreen] ({5+2*cos(90+67.5)},{2*sin(90+67.5)})--({5+2*cos(90+247.5)}, {2*sin(90+247.5)});
         \draw[thick,ForestGreen] ({5+2*cos(135+67.5)},{2*sin(135+67.5)})--({5+2*cos(135+247.5)}, {2*sin(135+247.5)});

         \draw [thick,cyan, bend right] ({5+1.63*cos(90)},{1.63*sin(90)}) to ({5+1.7*cos(56.25)},{1.7*sin(56.25)});
         \draw [thick,cyan, bend left] ({5+1.63*cos(180)},{1.63*sin(180)}) to ({5+1.7*cos(213.75)},{1.7*sin(213.75)});
         \draw [thick,cyan, bend left] ({5+1.63*cos(0)},{1.63*sin(0)}) to ({5+1.7*cos(33.75)},{1.7*sin(33.75)});
         \draw [thick,cyan, bend right] ({5+1.63*cos(270)},{1.63*sin(270)}) to ({5+1.7*cos(236.25)},{1.7*sin(236.25)});

         \draw [thick,cyan, bend right] ({5+1.63*cos(45+90)},{1.63*sin(45+90)}) to ({5+1.7*cos(45+56.25)},{1.7*sin(45+56.25)});
         \draw [thick,cyan, bend left] ({5+1.63*cos(45+180)},{1.63*sin(45+180)}) to ({5+1.7*cos(45+213.75)},{1.7*sin(45+213.75)});
         \draw [thick,cyan, bend left] ({5+1.63*cos(45+0)},{1.63*sin(45+0)}) to ({5+1.7*cos(45+33.75)},{1.7*sin(45+33.75)});
         \draw [thick,cyan, bend right] ({5+1.63*cos(45+270)},{1.63*sin(45+270)}) to ({5+1.7*cos(45+236.25)},{1.7*sin(45+236.25)});

         \draw [thick,cyan, bend right] ({5+1.63*cos(90+90)},{1.63*sin(90+90)}) to ({5+1.7*cos(90+56.25)},{1.7*sin(90+56.25)});
         \draw [thick,cyan, bend left] ({5+1.63*cos(90+180)},{1.63*sin(90+180)}) to ({5+1.7*cos(90+213.75)},{1.7*sin(90+213.75)});
         \draw [thick,cyan, bend left] ({5+1.63*cos(90+0)},{1.63*sin(90+0)}) to ({5+1.7*cos(90+33.75)},{1.7*sin(90+33.75)});
         \draw [thick,cyan, bend right] ({5+1.63*cos(90+270)},{1.63*sin(90+270)}) to ({5+1.7*cos(90+236.25)},{1.7*sin(90+236.25)});

         \draw [thick,cyan, bend right] ({5+1.63*cos(135+90)},{1.63*sin(135+90)}) to ({5+1.7*cos(135+56.25)},{1.7*sin(135+56.25)});
         \draw [thick,cyan, bend left] ({5+1.63*cos(135+180)},{1.63*sin(135+180)}) to ({5+1.7*cos(135+213.75)},{1.7*sin(135+213.75)});
         \draw [thick,cyan, bend left] ({5+1.63*cos(135+0)},{1.63*sin(135+0)}) to ({5+1.7*cos(135+33.75)},{1.7*sin(135+33.75)});
         \draw [thick,cyan, bend right] ({5+1.63*cos(135+270)},{1.63*sin(135+270)}) to ({5+1.7*cos(135+236.25)},{1.7*sin(135+236.25)});

          \draw[bend left] ({1.62*cos(0)},{1.62*sin(0)}) to ({1.62*cos(45)},{1.62*sin(45)});
          \draw [bend left] ({1.62*cos(45)},{1.62*sin(45)}) to ({1.62*cos(90)},{1.62*sin(90)});
          \draw [bend left] ({1.62*cos(90)},{1.62*sin(90)}) to ({1.62*cos(135)},{1.62*sin(135)});
          \draw [bend left] ({1.62*cos(135)},{1.62*sin(135)}) to ({1.62*cos(180)},{1.62*sin(180)});
          \draw [bend left] ({1.62*cos(180)},{1.62*sin(180)}) to ({1.62*cos(225)},{1.62*sin(225)});
          \draw [bend left] ({1.62*cos(225)},{1.62*sin(225)}) to ({1.62*cos(270)},{1.62*sin(270)});
          \draw [bend left] ({1.62*cos(270)},{1.62*sin(270)}) to ({1.62*cos(315)},{1.62*sin(315)});
          \draw [bend left] ({1.62*cos(315)},{1.62*sin(315)}) to ({1.62*cos(0)},{1.62*sin(0)});

          \draw ({1.62*cos(0)},{1.62*sin(0)})--({1.62*cos(180)},{1.62*sin(180)});
          \draw ({1.62*cos(45)},{1.62*sin(45)})--({1.62*cos(225)},{1.62*sin(225)});
          \draw ({1.62*cos(90)},{1.62*sin(90)})--({1.62*cos(270)},{1.62*sin(270)});
          \draw ({1.62*cos(135)},{1.62*sin(135)})--({1.62*cos(315)},{1.62*sin(315)});

         \draw (0,-2.5) node{$(i)$} (5,-2.5) node{$(ii)$};
\end{tikzpicture}
    \caption{$(i)$ Arcs with same color (except black) project onto a second systolic geodesic. $(ii)$ All second systolic geodesics on Bolza Surface.}
    \label{Fif4.1}
\end{figure}

\bibliographystyle{plain}
\bibliography{bibliography.bib}

\end{document}